\documentclass[a4paper]{article}
\usepackage[utf8]{inputenc}

\usepackage{subfig}
\usepackage{caption}
\usepackage{amsmath}
\usepackage{graphicx}
\usepackage[a4paper,top=2cm,bottom=2cm,left=3cm,right=3cm,marginparwidth=1.75cm]{geometry}
\usepackage{amssymb}
\usepackage{enumerate}
\usepackage{authblk}
\usepackage{amsthm}
\usepackage{subfig}
\usepackage{bbm}
\usepackage{verbatim}

\theoremstyle{definition}
\newtheorem{definition}{Definition}

\newtheorem{theorem}{Theorem}
\newtheorem*{theorem*}{Theorem}

\newtheorem*{conjecture*}{Conjecture}
\newtheorem{lemma}{Lemma}
\newtheorem*{lemma*}{Lemma}
\newtheorem{proposition}{Proposition}

\newtheorem{corollary}{Corollary}

\newtheorem*{remark}{Remark}
\newtheorem{example}{Example}
\newtheorem{question}{Question}

\newcommand{\Z}{\mathbb{Z}}
\newcommand{\F}{\mathbb{F}}
\newcommand{\Fp}{\mathbb{F}_p}

\newcommand{\N}{\mathbb{N}}
\newcommand{\C}{\mathbb{C}}
\newcommand{\ZnZ}{\mathbb{Z}\slash n\mathbb{Z}}

\newcommand{\modU}{\overline{U}}
\newcommand{\modV}{\overline{V}}
\newcommand{\rank}{\mathrm{rank}}
\newcommand{\eps}{\epsilon}

\newcommand{\vecx}{{\bf x}}

\title{Monodromy Groups of Dessins d'Enfant on Rational Polygonal Billiards Surfaces}

\author{
Richard A. Moy\\
Lee University\\
\texttt{rmoy@leeuniversity.edu}\\

\and

Jason Schmurr\\
Lee University\\
\texttt{jschmurr@leeuniversity.edu}\\

\and

Japheth Varlack\\
Wake Forest University\\
\texttt{varlja22@wfu.edu}
}

\begin{document}

\maketitle

\begin{abstract}
A \emph{dessin d’enfant}, or \emph{dessin}, is a bicolored graph embedded into a Riemann surface, and the monodromy group is an algebraic invariant of the dessin generated by rotations of edges about black and white vertices. A \emph{rational polygonal billiards surface} is a Riemann surface that arises from the dynamical system of billiards within a rational-angled polygon. In this paper, we compute the monodromy groups of dessins embedded into rational polygonal billiards surfaces and identify all possible monodromy groups arising from rational triangular billiards surfaces.
\end{abstract}

\section{Introduction}

In \cite{SMG19}, the authors investigated the connection between rational triangular billiards surfaces and Cayley graphs. In \cite{MMSV21}, the authors modified this approach by drawing {\it dessins d'enfant} on the rational triangular billiards surfaces and classifying their monodromy groups. In this paper, we generalize the main result in \cite{MMSV21} by computing the monodromy groups of {\it dessins d'enfant} drawn on billiard surfaces of {\it $k$-gons} with $k\geq 3$.

We show that all such monodromy groups can be expressed as the semidirect product $N\rtimes C_k$, where $N$ is isomorphic to the column span of a circulant matrix over $\mathbb
{Z}/n\mathbb{Z}$ for an appropriate integer $n$ (Theorem \ref{theorem:main1} and Lemma \ref{lemma:N}) and $C_k$ is the cyclic group of order $k$.

In Section \ref{section:structure_of_N}, we show how to use the Smith Normal Form to explicitly compute the monodromy group of any given rational billiards surface (Theorem \ref{theorem:main2}).

Next, for the case when $n=p$ for some prime $p$, we establish a correspondence between $k$-gons modulo $p$ and elements of $\mathbb{F}_p[x]$ which has the useful property that the monodromy group of the $k$-gon is completely determined by the greatest common divisor of the polynomial and $x^k-1$(Proposition \ref{proposition:main3}). This correspondence allows us to complete the classification of all monodromy groups of polygonal billiard surfaces for {\it $k$-gons} when $n=p$ is prime and $p>k$ (Theorem \ref{theorem:monodromy_groups}).

Finally, in Section \ref{section:composite_results}, we provide some preliminary results for composite $n$ which are sufficient to give a complete classification for triangles and an analogue of the main result in \cite{MMSV21} for quadrilaterals.

Throughout this paper, we will reference many well known algebraic and number theoretic results. See any introductory graduate abstract algebra book, such as \cite{A09}, or number theory book, such as \cite{L14}, for a reference.

\section{Background}\label{section:background}

\subsection{The Rational Billiard Surface Construction}

A rational billiards surface is constructed by gluing together copies of a polygon that result from consecutive reflections across the sides. This name is motivated by the task of examining the paths of balls that bounce around the interior of a billiard table. When a ball hits a side of the table, the resulting bounce is instead represented by gluing a reflection of the table across that side and continuing the billiard path in the reflected copy in the same direction. This way, the path of a ball is represented by a single geodesic on a flat surface instead of a jagged path that may cross back on itself. Equipped with this intuition, a rational billiards surface is constructed from all of the reflections required to account for every possible path a ball could take.

More formally, a rational billiard surface can be constructed from a $k$-gon $P$ whose angles are rational multiples of $\pi$, in the following way. Label the sides of $P$ as $e_0,\ldots e_{k-1}$, in consecutive counterclockwise order around $P$. Label the angles of $P$ as $\theta_i=\dfrac{a_i\pi}{n}$, where $\theta_i$ is the internal angle formed by sides $e_{i}$ and $e_{i+1}$ and $n\in \N$ is the least common denominator for the various $\dfrac{a_i}{n}$. Let $\Gamma$ be the dihedral group generated by the reflections $r_0,\ldots,r_{k-1}$ across lines through the origin parallel to the corresponding sides of $P$. This group consists of $2n$ elements \cite{AI88}, consisting of $n$ Euclidean rotations and $n$ Euclidean reflections. The rotation subgroup of $\Gamma$ is generated by rotation by the angle $\dfrac{2\pi}{n}$. Hence we may label the rotations using the notation $\rho_m$ for rotation by an angle of $\dfrac{2m\pi}{n}$. Let $\mathcal{P}=\{\gamma( P):\gamma\in\Gamma\}$. For each $\gamma (P)\in \mathcal{P}$ and each $r_i$, we glue together $\gamma(P)$ and $\gamma r_i (P)$ along their copies of $e_i$. The resulting object is called a \emph{translation surface}, since it is a Riemann surface whose change-of-coordinate maps are translations. See \cite{vorobets1996planar} and \cite{ZK}  for a detailed description of the rational billiards construction.

\subsection{Defining a Monodromy Group on the Surface}
Next, we draw a graph on this surface by placing a vertex in the center of each copy of $P$ and labeling it with the corresponding element of $\Gamma$. We draw an edge between two vertices $\alpha$ and $\beta$ precisely when $\alpha=\beta r_i$ for some $i$. This graph is the Cayley graph for $\Gamma$ with generating set $r_0,\ldots,r_{k-1}$. See \cite{SMG19} for a more in-depth exposition on this graph.

Since the generating set consists of reflections, this graph is bipartite, where one partite vertex set is the set of Euclidean rotations in $\Gamma$ and the other partite vertex set is the set of Euclidean reflections in $\Gamma$.

We will define a labeling scheme, introduced in \cite{MMSV21}, for the edges of the graph in following way. Take an arbitrary edge of the graph; one endpoint will be a vertex labeled $\rho_m$ and the other endpoint will be $\rho_m r_i$, for integers $m$ and $i$. We label this edge with the ordered pair $(m,i)\in C_n\times C_k$ where $C_n\times C_k$ is viewed as a set and not a group. (Here, $C_n$ represents the cyclic group of order $n$.) In fact this defines a bijection between the edge set of the graph and $C_n\times C_k$.

We can define a \emph{dessin d'enfant} on the surface by assigning a color to each of the partite sets (say, black for rotation and white for reflection) and by defining a cyclic ordering of the edges (oriented counterclockwise) around each vertex \cite{LZ04}. The ordering around a black vertex $\rho_m$ is $(m,0),(m,1),\ldots,(m,k-1)$, and the ordering around a white vertex $\rho_m r_i$ is $(m,i),(m-a_{i-1},i-1),(m-a_i-a_{i-1},i-2),\ldots,(m+a_{i+1},i+1)$. See Figure \ref{fig:basic_rotations}.

The ordering around a black vertex is apparent from our labeling scheme. To justify the ordering around a white vertex, observe that $r_{i+1}r_i=\rho_{-a_i}$ and $\rho_a\rho_b=\rho_{a+b}$, by basic facts about the composition of Euclidean reflections and rotations \cite{SMG19}. 
See Figure \ref{fig:111TriangleExample} for an example of this construction for the equilateral triangle, and see \cite{MMSV21} for further exposition on triangular billiards surfaces.

The \emph{monodromy group} of this dessin is a group $\langle \sigma_0,\sigma_1\rangle$ of permutations of the edges generated by two permutations $\sigma_0$ and $\sigma_1$. We define $\sigma_0$ to be the permutation that takes each edge to the next edge in the cyclic ordering about its black vertex. Similarly, we define $\sigma_1$ to be the permutation that takes each edge to the next edge in the cyclic ordering about its white vertex.

Therefore, we have that for any edge $(m,i)$,

\begin{equation}
\sigma_0[(m,i)]=(m,i+1) \label{sigma_0_def}\end{equation}

and

\begin{equation}
\sigma_1[(m,i)]=(m-a_{i-1},i-1).
\label{sigma_1_def}
\end{equation}


\begin{figure}
    \centering
    \includegraphics[scale=0.8]{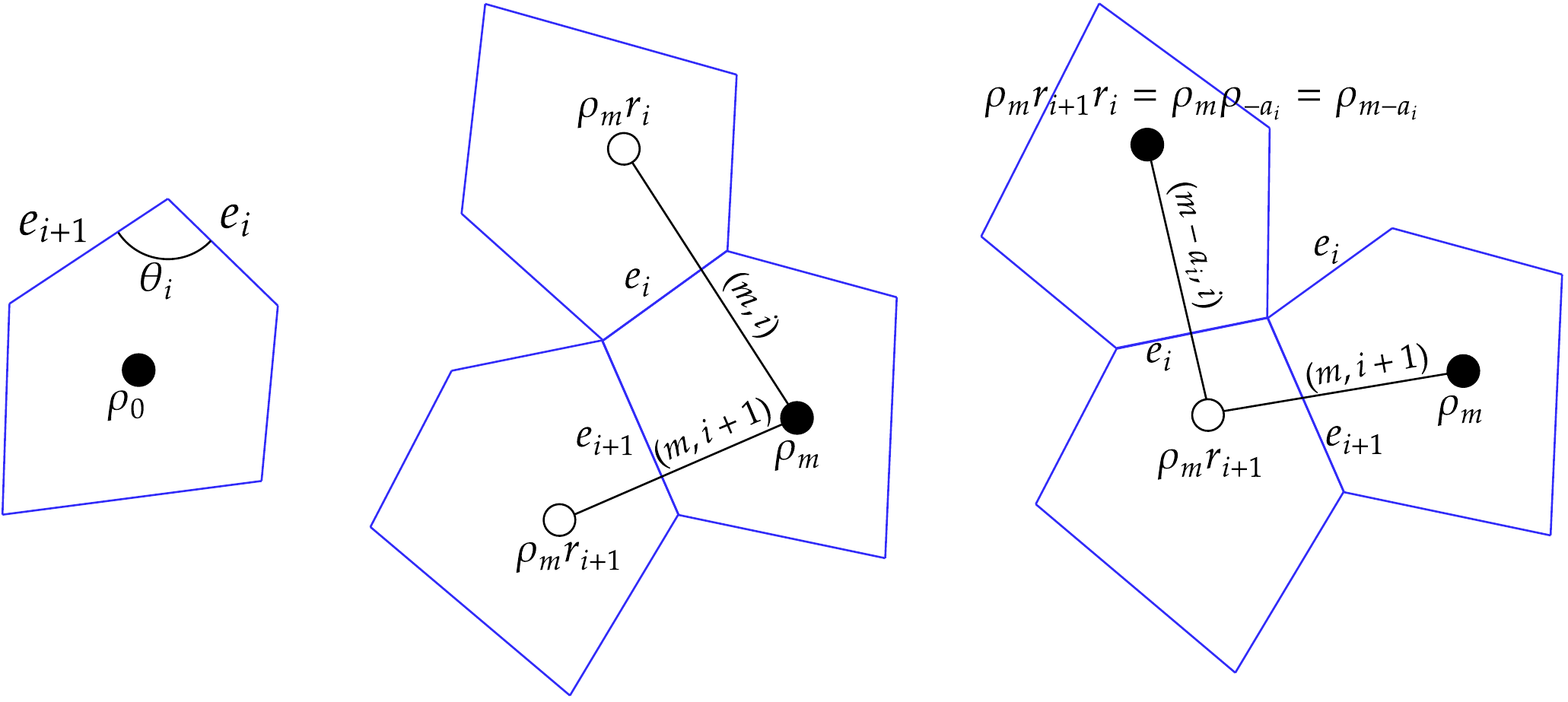}
    \caption{}
    \label{fig:basic_rotations}
\end{figure}



\begin{figure}
    \centering
    \includegraphics[scale=0.8]{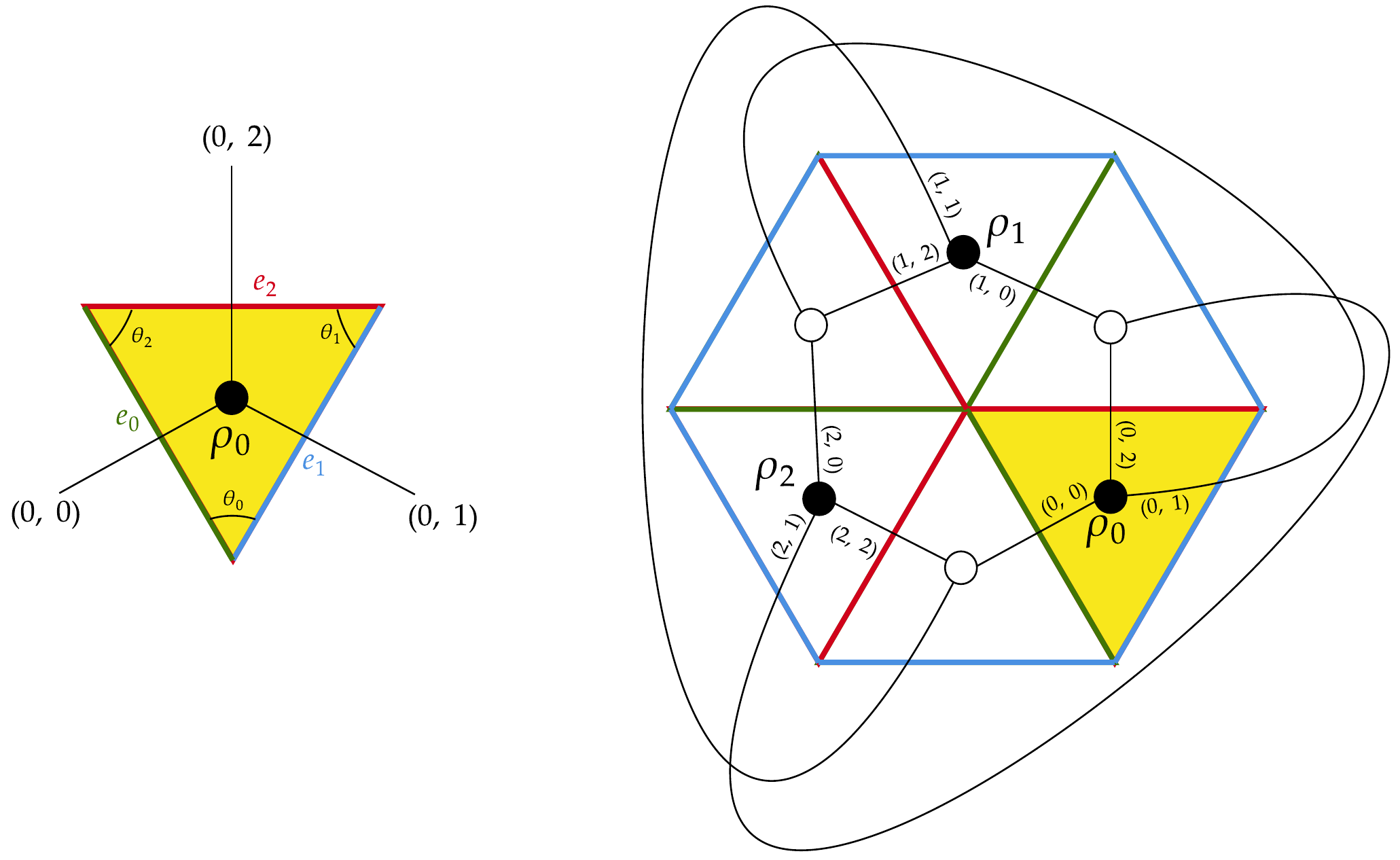}
    \caption{}
    \label{fig:111TriangleExample}
\end{figure}

\subsection{Representing Polygons by $k$-tuples}
Let $P$ be a rational polygon with consecutive internal angles $\dfrac{a_i\pi}{n}$, where $a_0+\ldots+a_{k-1}=(k-2)n$ and $\gcd(a_0,\ldots,a_{k-1},n)=1$. We shall use the notation $[a_0,a_1,\ldots,a_{k-1}]$ to represent $P$. Although this notation does not uniquely define $P$ up to geometric similarity when $k>3$, it does uniquely define the dessin drawn on $P$ up to graph isomorphism. This motivates the following definition.

\begin{definition}
If $k,n\in\N$ with $k\ge 3$, then an ordered $k$-tuple of positive integers $[a_0,\dots,a_{k-1}]$ {\it represents a geometric polygon, or geometric $k$-gon, modulo $n$} when $a_0+\dots+a_{k-1}=(k-2)n$ and $a_i<2n,\ a_i\ne n$ for all $i$ and $\gcd(a_0,\dots,a_{k-1},n)=1$. Throughout this paper, we will regularly use the term {\it $k$-gon} to refer to a geometric $k$-gon.
\end{definition}

\begin{remark}
The angles of a $k$-gon represented by $[a_0,\dots,a_{k-1}]$ modulo $n$ are $\frac{a_0}{n}\pi,\dots,\frac{a_{k-1}}{n}\pi$.
\end{remark}

It is not obvious that every $k$-tuple $[a_0,\dots,a_{k-1}]$ that represents a polygon modulo $n$ corresponds to a polygon in the plane with zero crossings. However, it is in fact true.

\begin{proposition}[Theorem 1, \cite{EFKT}]\label{proposition:polygon}
Suppose that $\theta_0,\dots,\theta_{k-1}$ is a sequence of angles (in radians) in the set $(0,\pi)\cup(\pi,2\pi)$. If $\theta_0+\dots+\theta_{k-1}=(k-2)\pi$, then there exists a polygon in the plane with no crossings with angles $\theta_0,\dots,\theta_{k-1}$ in that sequence.
\end{proposition}

Using this same convention, if the polygon $P$ is represented by $[a_0,a_1,\ldots,a_{k-1}]$ then we will use the notation $X(a_0,\ldots,a_{k-1})$ for the rational billiards surface arising from $P$ and $D(a_0,\ldots,a_{k-1})$ to represent the dessin drawn on $X(a_0,\ldots,a_{k-1})$. Finally, we will use $G(a_0,\ldots,a_{k-1})$ to represent the monodromy group of that dessin.

\section{Semidirect Product Structure of the Monodromy Group}\label{section:monodromy_groups} 

The goal of this section is to describe the monodromy groups as semidirect products of abelian groups.

\begin{theorem}\label{theorem:main1}
Let $[a_0,\ldots, a_{k-1}]$ represent a $k$-gon modulo $n$. Let $G(a_0,\ldots,a_{k-1}) = \langle \sigma_0, \sigma_1 \rangle$ be the monodromy group of the dessin $D(a_0,\dots,a_{k-1})$ drawn on the rational polygonal billiards surface $X(a_0,\dots,a_{k-1})$. Setting $N=\langle \sigma_0^x\sigma_1^x:\ 0< x<k \rangle$ and $H=\langle \sigma_0 \rangle$, we have $G(a_0,\ldots,a_{k-1}) = N \rtimes H$. 
\end{theorem}

\begin{lemma}\label{elements_commute}
The permutations $\sigma_0^x\sigma_1^x$ and $\sigma_0^y\sigma_1^y$ commute.
\end{lemma}

\begin{proof}

Let $(m,i)\in C_n\times C_k$ be an arbitrary edge of the dessin. 

From \eqref{sigma_0_def} and \eqref{sigma_1_def} we have that \begin{align}\sigma_0^x\sigma_1^x[(m,i)]=\sigma_0^x\left[\left(m-\sum_{j=1}^{x}a_{i-j},i-x\right)\right]=\left(m-\sum_{j=i-x}^{i-1}a_{j},i\right).\label{eq:commute}\end{align}

Therefore, $$\sigma_0^y\sigma_1^y\sigma_0^x\sigma_1^x[(m,i)]=\sigma_0^y\sigma_1^y\left[\left(m-\sum_{j=i-x}^{i-1}a_j,i\right)\right]=\left(m-\left(\sum_{j=i-x}^{i-1}a_{j}+\sum_{j=i-y}^{i-1}a_j\right),i\right).$$

Finally, $$\sigma_0^x\sigma_1^x\sigma_0^y\sigma_1^y[(m,i)]=\sigma_0^x\sigma_1^x\left[\left(m-\sum_{j=i-y}^{i-1}a_j,i\right)\right]=\left(m-\left(\sum_{j=i-x}^{i-1}a_{j}+\sum_{j=i-y}^{i-1}a_j\right),i\right),$$ establishing commutativity.
\end{proof}

\begin{definition}\label{def:N}
Let $N=\langle \sigma_0^x\sigma_1^x:\ 0<x<k \rangle.$ Observe that $\sigma_1^y\sigma_0^y=(\sigma_0^{k-y}\sigma_1^{k-y})^{-1}$.
\end{definition}

\begin{lemma}\label{lemma:fixed_subgroup}
The subgroup $N$ is precisely the subgroup of $G(a_0,\ldots,a_{k-1})$ that fixes the second component of the coordinates $(m,i)$.
\end{lemma}

\begin{proof}
Let $N'$ be the collection of elements in $G(a_0,\ldots,a_{k-1})$ that fix the second component of $(m,i)$. Clearly the identity is an element of $N'$. If $g,h\in N'$ then  $g h$ and $g^{-1}$ also fix the second component of $(m,i)$. Hence, $N'$ is a subgroup of $G(a_0,\ldots,a_{k-1})$ and the formula for $\sigma_0^x\sigma_1^x$ in \eqref{eq:commute} shows that $\sigma_0^x\sigma_1^x\in N'$. Since $\sigma_0^x\sigma_1^x$ generate $N$ as $x$ ranges from $1$ to $k-1$, we see that $N\le N'$.

Every element in $G(a_0,\ldots,a_{k-1})$ (and thus in $N'$) can be written as a product $g=(\sigma_0^{x_1}\sigma_1^{y_1})\dots (\sigma_0^{x_n}\sigma_1^{y_n})$ of $n$ pairs of the form $\sigma_0^{x_i}\sigma_1^{y_i}$ where $x_i,y_i\in \Z$. We will show that $N'\le N$ by induction on $n$. If $g=\sigma_0^{x_1}\sigma_1^{y_1}\dots \sigma_0^{x_n}\sigma_1^{y_n}\in N'$, we know that $\sum{x_i}\equiv \sum{y_i}\mod k$ by \eqref{sigma_0_def} and \eqref{sigma_1_def}.

Base Case: $n=1$
In this case, we see that $x_1\equiv y_1\mod k$. Since the orders of $\sigma_0$ and $\sigma_1$ are both $k$, we can assume $x_1=y_1$. Furthermore, we can also assume that $0\le x_1<k$. Hence, $g\in N$.

Induction Step:
Suppose our theorem is true for $n\ge 1$ and consider $n+1$. That is, suppose $g=\sigma_0^{x_1}\sigma_1^{y_1}\dots \sigma_0^{x_{n+1}}\sigma_1^{y_{n+1}}\in N'$. Consider
$$
g'=(\sigma_0^{x_1}\sigma_1^{x_1})^{-1} g (\sigma_0^{y_{n+1}}\sigma_1^{y_{n+1}})^{-1}=\sigma_1^{y_1-x_1}\sigma_0^{x_2}\sigma_1^{y_2}\dots\sigma_0^{x_n}{\sigma_1^{y_n} \sigma_0^{x_{n+1}-y_{n+1}}}
$$
Since $g\in N'$ then $g'\in N'$ and $(g')^{-1}\in N'$. Let $z_1=y_{n+1}-x_{n-1}, z_2=-x_n,\dots,z_n=-x_2$ and $w_1=-y_m,\dots,w_{n-1}=-y_2,w_n=x_1-y_1$. Observe that $(g')^{-1}=\sigma_0^{z_1}\sigma_1^{w_1}\dots \sigma_0^{z_n}\sigma_1^{w_n}$. Thus by the induction hypothesis, $(g')^{-1}\in N$. Hence, $g'\in N$ and $g\in N$. By induction, we have proven the desired result.

\end{proof}

\begin{lemma}\label{lemma:normal}
$N$ is a normal subgroup of $G(a_0,\ldots,a_{k-1})$.
\end{lemma}
\begin{proof}
Since $N=\langle \sigma_0^x\sigma_1^x:\ 0<x<k \rangle$ and $G(a_0,\ldots,a_{k-1}) = \langle \sigma_0, \sigma_1 \rangle$, proving $N\lhd G(a_0,\ldots,a_{k-1})$ is equivalent to proving the following statements:
\begin{enumerate}
\item $\sigma_1(\sigma_0^x\sigma_1^x)\sigma_1^{-1}\in N$
\item $\sigma_0(\sigma_0^x\sigma_1^x)\sigma_0^{-1}\in N$
\end{enumerate}

To prove 1, observe that 
$$\sigma_1(\sigma_0^x\sigma_1^x)\sigma_1^{-1}=(\sigma_1\sigma_0)(\sigma_0^{x-1}\sigma_1^{x-1})=(\sigma_0^{k-1}\sigma_1^{k-1})^{-1}(\sigma_0^{x-1}\sigma_1^{x-1})\in N.$$

To prove 2, observe that $$\sigma_0(\sigma_0^x\sigma_1^x)\sigma_0^{-1}=(\sigma_0^{x+1}\sigma_1^{x+1})(\sigma_1^{k-1}\sigma_0^{k-1})=(\sigma_0^{x+1}\sigma_1^{x+1})(\sigma_0\sigma_1)^{-1}\in N.$$ 
\end{proof}

\begin{lemma}\label{lemma:intersection}
$N\cap H = \{id\}$
\end{lemma}

\begin{proof}
Recall that $N=\langle \sigma_0^x\sigma_1^x:\ 0<x<k \rangle$ and $H=\langle \sigma_0 \rangle$. Suppose the intersection of these groups is not trivial. Then there is an element in $N$ that is equal to $\sigma_0^\ell$ for some $0< \ell<k$. Observe that $\sigma_0^\ell(m,i)=(m,i+\ell)$ and thus does not fix the second component of the edge labels. However, $N$ is generated by elements that fix the second component of the edge labels \eqref{eq:commute}. Hence, we have reached a contradiction.
\end{proof}

\begin{lemma}\label{lemma:product}
$NH = G(a_0,\ldots,a_{k-1})$
\end{lemma}
\begin{proof}
Recall that $N=\langle \sigma_0^x\sigma_1^x:\ 0<x<k \rangle$, $H = \langle \sigma_0 \rangle$, and $G(a_0,\ldots,a_{k-1}) = \langle \sigma_0, \sigma_1 \rangle$. Since $N\lhd G(a_0,\ldots,a_{k-1})$ and $H\le G(a_0,\ldots,a_{k-1})$, we know that $NH\le G(a_0,\ldots,a_{k-1})$. Observe that $\sigma_0\in NH$ and $\sigma_1=(\sigma_0^{k-1}\sigma_1^{k-1})^{-1}\sigma_0^{-1}\in NH$. Because $NH$ contains the generators of $G(a_0,\ldots,a_{k-1})$, we conclude that $NH=G(a_0,\ldots,a_{k-1})$.
\end{proof}

Now we proceed with the proof of Theorem 1.

\begin{proof}[\textbf{Proof of Theorem 1}]
The group $G(a_0,\ldots,a_{k-1})$ is a semi direct product of subgroups $N$ and $H$ if and only if the three conditions are true:
\begin{center}
\begin{enumerate}[I.]
    \item $N \lhd G(a_0,\ldots,a_{k-1})$
    \item $N \cap H = \{id\}$
    \item $NH = G(a_0,\ldots,a_{k-1})$
\end{enumerate}
\end{center}
Conditions $I$, $II$, and $III$ are satisfied by Lemmas \ref{lemma:normal}, \ref{lemma:intersection}, and \ref{lemma:product} respectively. Therefore, $G(a_0,\ldots,a_{k-1})$ is a semidirect product of subgroups $N$ and $H$.
\end{proof}

\begin{remark}
The action of $H$ on $N$ in the semidirect product is via conjugation by elements of $H$.
\end{remark}

\section{Computing the Structure of $N$}\label{section:structure_of_N}
In this section, we prove several properties about the subgroup $N\triangleleft G(a_0,\ldots,a_{k-1})$, introduced in Definition \ref{def:N}, to provide more precise information about the structure of $N$ and, by extension, $G(a_0,\ldots,a_{k-1})$.

Let $S=\{\sigma_1^{-j}(\sigma_0^{-1}\sigma_1^{-1})\sigma_1^{j}:0\leq j <k\}$. We first show that one can generate $N$ using the elements of $S$.

\begin{lemma}\label{lemma:new_N}
The subgroup $N$ is generated by $S$.
\end{lemma}
\begin{proof}
Recall that $N=\langle \sigma_0^x\sigma_1^x:\ 0<x<k \rangle$. Let $S=\{\sigma_1^{-j}(\sigma_0^{-1}\sigma_1^{-1})\sigma_1^{j}:0\leq j <k\}$. We claim $\langle S\rangle=N.$ Using \eqref{sigma_0_def} and \eqref{sigma_1_def}, we see that $\sigma_1^{-j}(\sigma_0^{-1}\sigma_1^{-1})\sigma_1^{j}$ fixes the second component of the coordinates $(m,i)$ and is thus an element of $N$ by Lemma \ref{lemma:fixed_subgroup}. Hence, $\langle S \rangle\le N$.

We will prove that $\sigma_0^j\sigma_1^j\in \langle S\rangle$ using induction. Observe that $\sigma_1^{-1}(\sigma_0^{-1}\sigma_1^{-1})\sigma_1^1=(\sigma_0\sigma_1)^{-1}$. Hence, $\sigma_0\sigma_1\in\langle S\rangle$.

Suppose $\sigma_0^{j-1}\sigma_1^{j-1}\in \langle S\rangle$. Observe that  $\sigma_1^{-j}(\sigma_0^{-1}\sigma_1^{-1})\sigma_1^j=(\sigma_0^j\sigma_1^j)^{-1}\sigma_0^{j-1}\sigma_1^{j-1}$ which implies $\sigma_0^j\sigma_1^j\in\langle S\rangle$. Thus, $\sigma_0^j\sigma_1^j\in \langle S\rangle$ for all $j>0$ and hence $N\le \langle S\rangle$.
\end{proof}

As we observed in Lemma \ref{lemma:fixed_subgroup}, the subgroup $N$ is precisely the subgroup of $G(a_0,\ldots,a_{k-1})$ which fixes the second component of the edge $(m,i)$. Hence, we may view any element $g\in N$ as a column vector $\begin{bmatrix}x_0\\ \vdots \\x_{k-1}\end{bmatrix}\in (\mathbb{Z}/n\mathbb{Z})^k$, where $g(m,i)=(m+x_i,i)$. It follows from equations \eqref{sigma_0_def} and \eqref{sigma_1_def} that $\sigma_1^{-j}(\sigma_0^{-1}\sigma_1^{-1})\sigma_1^{j}(m,i)=(m+a_{i-j},i)$. 
Therefore the set $S=\{\sigma_1^{-j}(\sigma_0^{-1}\sigma_1^{-1})\sigma_1^{j}:0\leq j <k\}$ can be identified with the columns of the  matrix 

\begin{equation}\label{equation:C}
C=\begin{bmatrix}
a_0&a_{k-1}&\dots& a_2 &a_1\\
a_1& a_0&a_{k-1}& &a_2\\
\vdots& a_1&a_0&\ddots & \vdots\\
a_{k-2}&  & \ddots &\ddots&a_{k-1}\\
a_{k-1}&a_{k-2}&\dots&a_1&a_0
\end{bmatrix}
\end{equation}
in $M_k(\ZnZ)$ where $M_k(\ZnZ)$ is the set of $k\times k$ matrices with entries in $\ZnZ$. We make this statement more formal in the following lemma.

\begin{lemma}\label{lemma:N}
The subgroup $N$ is isomorphic to the span of the columns of $C$.  
\end{lemma} 

\begin{proof}
From \eqref{sigma_0_def}, \eqref{sigma_1_def}, and Lemma \ref{lemma:fixed_subgroup}, we see that an arbitrary element $g\in N$ has the form $g(m,i)=(m+x_i,i)$ where $\vecx=\begin{bmatrix} x_0\\ \vdots \\ x_{k-1}\end{bmatrix}\in(\ZnZ)^k$. We define a homomorphism $\varphi:N\rightarrow (\ZnZ)^k$ via $\varphi(g)=\vecx$. It is easy to check that $\varphi$ is a well-defined map with $\varphi(g_1g_2)=\varphi(g_1)+\varphi(g_2)$. 

It is also easy to see that $\varphi$ is injective. If $\varphi(g)=\begin{bmatrix}0\\ \vdots \\ 0\end{bmatrix}$, then $g$ fixes every edge of the dessin. Hence, $g$ is the identity element since the monodromy group acts faithfully on the edges of the dessin. Thus, we may conclude that $\varphi$ maps $N$ bijectively onto $\varphi(N)$.

Since the elements of the set $S$ generate $N$, we conclude that the set of vectors of the form $\varphi(\sigma_1^{-j}(\sigma_0^{-1}\sigma_1^{-1})\sigma_1^j)=\begin{bmatrix}a_{k-j}\\ \vdots \\ a_{k-j-1}\end{bmatrix}$ where $0\le j<k$ spans $\varphi(N)$. And thus, $N$ is isomorphic to the span of the columns of $C$.

\end{proof}

\begin{remark}
It is worth noting that when viewing $N$ as a set of vectors in $(\ZnZ)^k$, there is a natural group action of $C_k \cong H$ on $N$ which is the cyclic permutation of the vector entries. That is, the homomorphic image of $H$ in Aut$(N)$ is precisely the subgroup of cyclic permutations of vector entries.
\end{remark}

In order to determine the group structure of $N$, we will use row and column operations on the matrix $C$.

\subsection{Smith Normal Form}\label{section:snf}
In previous sections we establish that the monodromy group $G(a_0,\ldots,a_{k-1})$ can be expressed as the semidirect product of $C_k$ and some finite abelian subgroup $N$,
where $N$ has a natural $\mathbb{Z}/n\mathbb{Z}$-module structure. In this section we explore the explicit computation of $N$. This can be done via
the \emph{Smith Normal Form}. See \cite{A09} or \cite{MR09} for a reference.

\begin{definition}\label{def:SNF} 
The {\it Smith Normal Form} of a matrix $A$ with entries from a ring $R$ is a factorization $A=UDV$ where
\begin{itemize}
    \item $D=\begin{bmatrix}d_1& & \\ & \ddots &\\ & & d_k \end{bmatrix}$ is a diagonal matrix
    \item $d_i|d_{i+1}$ for all $i$
    \item $U$ and $V$ are square matrices with determinant $\pm 1$
\end{itemize}
\end{definition}

Consider the $R$-module $M$, which is a submodule of $R^k$, generated by the columns of $A$. Then as a group, $M$ is isomorphic to the direct product 
$$d_1 R\times\cdots\times d_k R.$$ 

The elements $d_1\ldots,d_k$ are called the \emph {elementary divisors} of $M$. In \cite{K49}, Kaplansky defines an \emph{elementary divisor ring} $R$ to be a ring
over which all matrices have a Smith Normal Form. It is well-known (see \cite{K49}) that all PID's are elementary divisor rings. However, not all elementary divisor rings are domains. Indeed, it follows from Corollary 2.3 of \cite{LLS74} that $\mathbb{Z}/n\mathbb{Z}$ is an elementary divisor ring. Hence, we can always compute the group structure of one of our particular monodromy groups by computing the Smith Normal Form of the associated
circulant matrix. 

In practice, algorithms exist for computing the Smith Normal Form of a matrix over $\mathbb{Z}$. Therefore, to compute the Smith Normal Form of a matrix over 
$\mathbb{Z}/n\mathbb{Z}$, it is convenient to compute the Smith Normal Form of an associated matrix over $\mathbb{Z}$ and then apply the standard ring homomorphism to 
reduce modulo $n$. 

Since the matrices $U$ and $V$ in Definition \ref{def:SNF} are invertible over $\Z$, their reductions modulo $n$ (call them $\modU$, $\modV$) are invertible over $\ZnZ$. Therefore, the transformation $x\mapsto \modU^{-1}\cdot x$ is an isomorphism from $(\ZnZ)^k\mapsto(\ZnZ)^k$.

Hence, the $\ZnZ$ submodule generated by $v_0 \modV^{-1},\dots,v_{k-1} \modV^{-1}$ is isomorphic to the $\ZnZ$ submodule generated by the columns of $D$ which are 
$$
\modU^{-1}\overline{v_1}\modV^{-1}=\begin{bmatrix}\overline{d_1}\\0\\ \vdots \\ 0\end{bmatrix},\ \dots,\ \modU^{-1} \overline{v_k}\modV^{-1}=\begin{bmatrix}0\\ \vdots \\ 0 \\\overline{d_k}\end{bmatrix}.
$$
 
Hence, $N$ is isomorphic to $\overline{d}_1\ZnZ\oplus\dots\oplus \overline{d}_k\ZnZ$ where $\overline{d}_i$ is the reduction of $d_i$ modulo $n$. And therefore, 
$$
N\cong \bigoplus_{i=1}^k{\Z\slash \delta_i\Z}
$$
where $\delta_i=\frac{n}{\gcd(d_i,n)}$. We summarize these results with the following theorem, combining the results from Theorem \ref{theorem:main1}.

\begin{theorem}\label{theorem:main2}
Let $C$ be the matrix defined in \eqref{equation:C} and let $d_1,\dots,d_k$ be the elementary divisors of $C$ coming from its Smith Normal Form when viewing $C$ as a matrix over $\Z$. Then $$G(a_0,\ldots,a_{k-1})=\left(\bigoplus_{i=1}^k{C_{\delta_i}}\right)\rtimes C_k$$ where $\delta_i=\frac{n}{\gcd(d_i,n)}$.
\end{theorem}

Note that some of the $\delta_i$ may equal $1$, in which case the group $C_{\delta_i}$ is trivial.

\begin{example} Consider the quadrilateral with angles $(\frac{2}{5}\pi, \frac{2}{5}\pi, \frac{2}{5}\pi, \frac{4}{5}\pi)$. This gives the billiards surface $X(2,2,2,4)$ and dessin $D(2, 2, 2, 4)$. To calculate the monodromy group $G(2, 2, 2, 4)$ of the dessin, we compute the smith normal form for the circulant matrix
$$
C=\begin{bmatrix}
2 & 4 & 2 & 2 \\
2 & 2 & 4 & 2 \\
2 & 2 & 2 & 4 \\
4 & 2 & 2 & 2
\end{bmatrix}
=UDV=
\begin{bmatrix}
-11 & -12 & -14 & -3 \\
-11 & -12 & -13 & -3 \\
-7 & -8 & -9 & -2 \\
-11 & -13 & -14 & -3
\end{bmatrix}
\begin{bmatrix}
2 & 0 & 0 & 0 \\
0 & 2 & 0 & 0 \\
0 & 0 & 2 & 0 \\
0 & 0 & 0 & 10
\end{bmatrix}
\begin{bmatrix}
1 & 0 & 0 & 4 \\
-1 & 1 & 0 & 0 \\
0 & -1 & 1 & 0 \\
0 & 0 & -1 & -3
\end{bmatrix}
$$
where $U$ and $V$ are unimodular. This gives us
\begin{align*}
    \delta_1 = \delta_2 = \delta_3 = \frac{5}{\text{gcd}(2, 5)} = 5,\ \ \delta_4 = \frac{5}{\text{gcd}(10, 5)} = 1.
\end{align*}
Then we have
$$
G(2, 2, 2, 4) = \left(C_5\times C_5\times C_5\right) \rtimes C_4.
$$
\end{example}

As a consequence of Theorem 2, one can quickly compute the monodromy groups of any rational triangular billiards surfaces.

\begin{corollary}[Theorem 1, \cite{MMSV21}]
Let $[a_0, a_1, a_2]$ represent a triangle modulo $n$.
Let $G(a_0, a_1, a_2)= \langle \sigma_0, \sigma_1 \rangle$ be the monodromy group of the dessin $D(a_0,a_1,a_2)$ drawn on the triangular billiards surface $X(a_0,a_1,a_2)$. Setting $N=\langle \sigma_0\sigma_1, \sigma_0^2\sigma_1^2 \rangle$ and $H=\langle \sigma_0 \rangle$, we have $G(a_0, a_1, a_2) = N \rtimes H$. Furthermore, if $n = a_0 + a_1 + a_2$ and $\alpha = \gcd(n, a_0a_1 - a_2^2)$, then
$$
G(a_0, a_1, a_2) \cong (C_n \times C_{\frac{n}{\alpha}}) \rtimes C_3.
$$
\end{corollary}

\begin{proof}
Consider the arbitrary rational triangle with angles $\left(\dfrac{a_0\pi}{n},\dfrac{a_1\pi}{n},\dfrac{a_2\pi}{n}\right)$, where the $a_i$ are positive integers, $a_0+a_1+a_2=n$, and $\gcd(a_0,a_1,a_2,n)=1$. Observe that it follows that $\gcd(a_0,a_1,n)=1$ as well. The normal subgroup $N$ of the associated monodromy group is represented by the column span of $C=\begin{bmatrix}
a_0 & a_1 & a_2 \\
 a_1 & a_2 & a_0 \\
a_2 & a_0 & a_1 
\end{bmatrix}$
 over $\mathbb{Z}/n\mathbb{Z}$.

Since $\gcd(a_0,a_1,n)=1$, there exist integers $s$, $t$, and $u$ such that $sa_0+ta_1+un=1$, and hence $sa_0+ta_1\equiv 1 \mod n$.

 We can perform the following transformations of $C$ by applying row and column operations (working modulo $n$):

 $$C=\begin{bmatrix}
a_0 & a_1 & a_2 \\
 a_1 & a_2 & a_0 \\
a_2 & a_0 & a_1 
\end{bmatrix}\rightarrow
\begin{bmatrix}
a_0 & a_1 & a_2 \\
 a_1 & a_2 & a_0 \\
0 & 0 & 0 
\end{bmatrix}\rightarrow
\begin{bmatrix}
a_0 & a_1 & 0 \\
 a_1 & a_2 & 0 \\
0 & 0 & 0 
\end{bmatrix}
\rightarrow
\begin{bmatrix}
1 & sa_1+ta_2 & 0 \\
 0 & -a_1^2+a_0a_2 & 0 \\
0 & 0 & 0 
\end{bmatrix}
\rightarrow
\begin{bmatrix}
1 & 0 & 0 \\
 0 & -a_1^2+a_0a_2 & 0 \\
0 & 0 & 0 
\end{bmatrix}.
$$

This yields the factorization

$$\begin{bmatrix}
1 & 0 & 0 \\
 0 & -a_1^2+a_0a_2 & 0 \\
0 & 0 & 0 
\end{bmatrix}=\begin{bmatrix}
s & t & 0 \\
 -a_1 & a_0 & 0 \\
0 & 1 & 1 
\end{bmatrix}
\begin{bmatrix}
a_0 & a_1 & a_2 \\
 a_1 & a_2 & a_0 \\
a_2 & a_0 & a_1 
\end{bmatrix}
\begin{bmatrix}
1 & 0 & 1 \\
 -sa_1+ta_2 & 1 & 1 \\
0 & 0 & 1 
\end{bmatrix}.
$$
Or, equivalently,
$$C=\begin{bmatrix}
a_0 & a_1 & a_2 \\
 a_1 & a_2 & a_0 \\
a_2 & a_0 & a_1 
\end{bmatrix}=
\begin{bmatrix}
a_0 & -t & 0 \\
 a_1 & s & 0 \\
-a_1 & -s & 1 
\end{bmatrix}
\begin{bmatrix}
1 & 0 & 0 \\
 0 & -a_1^2+a_0a_2 & 0 \\
0 & 0 & 0 
\end{bmatrix}
\begin{bmatrix}
1 & 0 & -1 \\
sa_1-ta_2 & 1 & -sa_1+ta_2-1 \\
0 & 0 & 1
\end{bmatrix}.
$$

One easily checks that the diagonalizing matrices are unimodular. It then follows from Theorem 2 that the monodromy group of the $(a_0,a_1,a_2)$ triangle is 
$$
  (C_n\times C_{n/\alpha})\rtimes C_3,  
$$
where $\alpha=\gcd(n,a_0a_2-a_1^2)$. 
\end{proof}

\begin{corollary}[Corollary to Theorem \ref{theorem:main2}]\label{corollary:triangle}
The monodromy group of the dessin drawn on the rational billiards surface of the regular $k$-gon is $C_{\frac{k}{\gcd(k,2)}}\times C_k$.

\end{corollary}

\begin{proof}
The angles of the regular $k$-gon are $\frac{k-2}{k}\pi$. When $k$ is odd, $a_0=\dots=a_{k-1}=k-2$ and $n=k$ since $\gcd(k-2,k)=1$. When $k$ is even, $a_0=\dots=a_{k-1}=\frac{k-2}{2}$ and $n=\frac{k}{2}$ since $\gcd(\frac{k-2}{2},\frac{k}{2})=1$.

Since $a_0=\dots=a_{k-1}$, we see that the matrix $C$ is a $k\times k$ matrix whose entries are all equal to $\frac{k-2}{\gcd(k,2)}$. We deduce that the Smith Normal Form matrix $D=\begin{bmatrix} \frac{k-2}{\gcd(k,2)} & 0 &\dots &0\\ 0 & 0 & \dots& 0\\ \vdots & \vdots& \ddots & 0\\ 0&0&0&0\end{bmatrix}$.

By Theorem \ref{theorem:main2}, it follows that $G(a_0,\ldots,a_{k-1})\cong C_{\frac{k}{\gcd(k,2)}}\rtimes C_k$. Since the subgroup $H\cong C_k$ acts on $N\cong C_{\frac{k}{\gcd(k,2)}}$ via cyclic permutation of the vector entries of $C$, we see that the the semidirect product action of $H$ on $N$ is trivial since all the columns of $C$ are identical. Hence, the semidirect product is actually a direct product.
\end{proof}

In \cite{H86}, Howell addresses the problem of computing the span of a set of vectors over $\Z/n\Z$. Howell considers a matrix $A$ with entries 
in $\Z/n\Z$. He then shows that $A$ can be reduced via elementary row operations to an upper triangular matrix $U$ whose rows have the same span as $A$. This matrix $U$ is known as \emph{Howell Normal Form}. However, Smith Normal Form has the advantage of directly computing the isomorphism class of the vector span as an abelian group, via the ordered list of elementary divisors.

\section{Algebraic Polygons}\label{section:algebraic_polygons}
In this section, we introduce the notion of an \emph{algebraic polygon} and develop the relevant theory with the goal of proving results about actual polygons. We arrive at the concept of an algebraic polygon by relaxing the constraints on polygons modulo $n$ slightly:

\begin{definition}
If $k,n\in\N$ with $k\ge 2$, then an ordered $k$-tuple of nonnegative integers $[a_0,\dots,a_{k-1}]$ {\it represents an algebraic polygon, or $k$-gon, modulo $n$} if $a_0+\dots+a_{k-1}\equiv 0\mod n$ and $\gcd(a_0,\dots,a_{k-1},n)=1$. Observe that $[0,\dots,0]$ is not an algebraic $k$-gon.
\end{definition}

Every geometric polygon modulo $n$ is also an algebraic polygon modulo $n$. We shall define a ``monodromy group'' for any algebraic polygon in a natural way which coincides with the monodromy groups associated to geometric polygons described in Section \ref{section:structure_of_N}. It turns out that it is relatively easy to classify the possible monodromy groups for all algebraic polygons modulo a prime $p$ (we do this in Theorem \ref{theorem:algebraic_monodromy_groups}). The challenge is to determine when, for a given monodromy group $G$ of an algebraic polygon, there exists a \emph{geometric} polygon with a monodromy group isomorphic to $G$. Lemmas \ref{lemma:associate} and \ref{lemma:associate_groups} show that this is always possible if none of the entries in the algebraic polygon are zero modulo $n$. This motivates work in Section \ref{section:proving_monodromy_groups} to produce algebraic polygons with nonzero entries.

\begin{remark}
Note that the definition of an algebraic polygon allows for an algebraic $2$-gon even though no geometric $2$-gons exist. Despite this fact, algebraic $2$-gons can be used to produce geometric $k$-gons via Proposition \ref{proposition:combine_groups}.
\end{remark}

\subsection{Results About Algebraic Polygons}

\begin{definition}
We say that two algebraic polygons, $[a_0,\dots,a_{k-1}]$ and $[b_0,\dots,b_{k-1}]$ modulo $n$ are {\it associates} if there exists $c\in (\ZnZ)^\times$ such that $b_i\equiv c a_i\pmod n$ for all $i$.
\end{definition}

\begin{remark}
Our definition of associate algebraic polygons coincides with the definition of associate triangles from Aurell and Itzykson \cite{AI88}.
\end{remark}

Observe that reflex angles lead to interesting associate polygons. For example, the (algebraic) polygons $[3,5,11,1]$ and $[3,15,1,1]$ are associates modulo $10$.

\begin{proposition}\label{proposition:adjust_angles}
Suppose that $[a_0,\dots,a_{k-1}]$ represents an algebraic polygon modulo $n$. Further suppose that $0<a_i<2n, a_i\ne n$ for all $i$ and $a_0+\dots+a_{k-1}\le(k-2)n$. Then there exists an associate polygon $[b_0,\dots,b_{k-1}]$. Consequently, there exists a polygon in the plane with consecutive angles $\frac{b_0}{n}\pi,\dots,\frac{b_{k-1}}{n}\pi$ and zero crossings.
\end{proposition}

Observe that Proposition \ref{proposition:adjust_angles} produces a polygon, not simply an algebraic polygon.

\begin{proof}
If $a_0+\dots+a_{k-1}=(k-2)n$, then, letting $a_i=b_i$, $[a_0,\dots,a_{k-1}]=[b_0,\dots,b_{k-1}]$ represents an associate polygon modulo $n$. If $a_0+\dots+a_{k-1}<(k-2)n$, then let $d=\frac{(k-2)n-(a_0+\dots+a_{k-1})}{n}$. We can find $a_{i_1},\dots,a_{i_d}$ with $i_1,\dots,i_d$ distinct such that $a_{i_j}<n$. Add $n$ to each of these $a_{i_j}$ to obtain $b_{i_j}= a_{i_j}+n\equiv a_{i_j}\mod n$. Let $b_{i}= a_{i}$ for all other indices $i\ne i_j$. Thus, $[b_0,\dots,b_{k-1}]$ represents an associate $k$-gon modulo $n$.  By Proposition \ref{proposition:polygon}, there exists a polygon in the plane with consecutive angles $\frac{b_0}{n}\pi,\dots,\frac{b_{k-1}}{n}\pi$ and zero crossings.
\end{proof}

\begin{example}
Consider the algebraic polygon $[1,2,2,7]$ modulo $12$. Using the procedure in Proposition \ref{proposition:adjust_angles}, we produce the associate geometric polygon $[13,2,2,7]$.
\end{example}

We will use the following lemma many times to verify that an algebraic $k$-gon satisfies the hypotheses of Proposition \ref{proposition:adjust_angles}.



\begin{lemma}\label{lemma:associate}
Suppose that $[a_0,\dots,a_{k-1}]$ is an algebraic polygon modulo $n$ with $a_i\not\equiv 0\mod n$ for all $i$. Then, $[a_0,\dots,a_{k-1}]$ has an associate $k$-gon $[b_0,\dots,b_{k-1}]$ that is a polygon modulo $n$. If $n=p$ is a prime and $p\ge k-1$, then there exists an associate {\it convex} $k$-gon $[b_0,\dots,b_{k-1}]$ that is a polygon modulo $p$.
\end{lemma}

\begin{proof}
Let $d_i=\gcd(a_i,n)$. Observe that the subgroup $\langle a_i: 0\le i< k\rangle$ of $\ZnZ$ is isomorphic to $\langle d_i\rangle$. Let $d=\min(d_i:0\le i<k)$. Without loss of generality, assume that $d=gcd(a_0,n)$. The proof is analogous in all other cases.

There exists $c\in (\ZnZ)^\times$ such that $\overline{ca_{0}}=d$ where $\overline{ca_{0}}$ is the reduction of $ca_0$ modulo $n$. Further observe that $\overline{ca_{j}}\le n-d_{j}\le n-d$ for $1\le j<k$ since $ca_{j}\not\equiv 0\mod n$ and $ca_{j}$ is a multiple of $d_{j}$. Hence, 
$$
\overline{ca_{0}}+\overline{ca_{1}}+\dots +\overline{ca_{k-1}}\le d+(n-d)+\dots +(n-d)<(k-1)n.
$$
Observe that 
$$\overline{ca_{0}}+\overline{ca_{1}}+\dots +\overline{ca_{k-1}}\equiv c(a_0+\dots +a_{k-1})\equiv 0\mod n.$$ Therefore, $\overline{ca_{0}}+\overline{ca_{1}}+\dots +\overline{ca_{k-1}}\le (k-2)n.$ Using Proposition \ref{proposition:adjust_angles}, we obtain the desired $[b_0,\dots,b_{k-1}]$.

Now consider the case where $n=p$ is a prime and $p\ge k-1$. Since $a_i\not\equiv 0\mod p$ for all $i$, the reduction of $a_i$ modulo $p$ can be chosen so that $0<\overline{a_i}<p$ for all $i$. Since $[a_0,\dots,a_{k-1}]$ is an algebraic polygon, we know that $a_0+\dots +a_{k-1}\equiv 0 \mod p$. Therefore, $\overline{a_0}+\dots+\overline{a_{k-1}}=cp$ where $0<c<k$. Choose $c'\in \Z\slash p\Z$ so that $c'\cdot c\equiv k-2\mod p$. We see that $c'\cdot \frac{\overline{a_0}+\dots+\overline{a_{k-1}}}{p}\equiv c'\cdot c\equiv k-2\mod p$. Hence, $\overline{c'a_0}+\dots \overline{c'a_{k-2}}=(k-2)p$ and thus, letting $b_i=\overline{c'a_i}$, $[b_0,\dots,b_{k-1}]$ is a $k$-gon modulo $p$. Since $0<b_i<p$ for all $i$, we see that $[b_0,\dots,b_{k-1}]$ represents a convex polygon.
\end{proof}

\subsection{Monodromy Groups of Algebraic Polygons}
The purpose of introducing algebraic polygons is to understand monodromy groups of actual polygons. Therefore, we must associate to each algebraic polygon a monodromy group that coincides with the monodromy group in Section \ref{section:background} for geometric polygons. 

\begin{definition}
The {\it monodromy group} associated with an algebraic $k$-gon $[a_0,\dots,a_{k-1}]$ modulo $n$ is the group $N\rtimes C_k$ where $N$ is the additive group generated by the columns of the matrix
$$C=\begin{bmatrix}
a_0&a_{k-1}&\dots& a_2 &a_1\\
a_1& a_0&a_{k-1}& &a_2\\
\vdots& a_1&a_0&\ddots & \vdots\\
a_{k-2}&  & \ddots &\ddots&a_{k-1}\\
a_{k-1}&a_{k-2}&\dots&a_1&a_0
\end{bmatrix}$$
in the $\ZnZ$ module $(\ZnZ)^k$. The group $C_k$ acts on the columns of $C$ by cyclicly permuting the entries of a vector. 
\end{definition}

The monodromy groups that arose in Section \ref{section:background} were monodromy groups of dessins d'enfant drawn on rational billiards surfaces. Although these surfaces and dessins do not exist for algebraic polygons, associating a monodromy group with them will still prove quite useful theoretically.

\begin{remark}
If $[a_0,\dots,a_{k-1}]$ is a $k$-gon modulo $n$, then its monodromy group above is the same as the monodromy group of $D(a_0,\dots,a_{k-1})$ drawn on the rational polygonal billiards surface $X(a_0,\dots,a_{k-1})$. See Sections \ref{section:background} and \ref{section:structure_of_N} for reference.
\end{remark}

The following lemma illustrates that the monodromy group of associate algebraic polygons are isomorphic.

\begin{lemma}\label{lemma:associate_groups}
Fix $n\in \N$. If $[a_0,\dots,a_{k-1}]$ and $[b_0,\dots,b_{k-1}]$ are associate algebraic polygons, then their monodromy groups are the same.
\end{lemma}

\begin{proof}
Since $[a_0,\dots,a_{k-1}]$ and $[b_0,\dots,b_{k-1}]$ are associates, there exists $c\in (\ZnZ)^\times$ such that $b_i\equiv c p_i$ for all $i$. Let $C'$ and $C''$ be the corresponding circulant matrices for $[b_0,\dots,b_{k-1}]$ and $[a_0,\dots,a_{k-1}]$ respectively. Therefore, 

\begin{align*}
C'=\begin{bmatrix}
b_0&b_{k-1}&\dots& b_2 &b_1\\
b_1& b_0&b_{k-1}& &b_2\\
\vdots& b_1&b_0&\ddots & \vdots\\
b_{k-2}&  & \ddots &\ddots&b_{k-1}\\
b_{k-1}&b_{k-2}&\dots&b_1&b_0
\end{bmatrix} 
 &\equiv c\begin{bmatrix}
a_0&a_{k-1}&\dots& a_2 &a_1\\
a_1& a_0&a_{k-1}& &a_2\\
\vdots& a_1&a_0&\ddots & \vdots\\
a_{k-2}&  & \ddots &\ddots&a_{k-1}\\
a_{k-1}&a_{k-2}&\dots&a_1&a_0
\end{bmatrix}
\equiv c\cdot C''\mod n.
\end{align*}
Since $C'$ and $C''$, are scalar multiples of each other by a unit, the spans of their columns are equal. The result follows.
\end{proof}

\begin{proposition}\label{proposition:combine_groups}
Suppose that $[a_0,\dots,a_{k-1}]$ and $[b_0,\dots,b_{k-1}]$ represent algebraic $k$-gons modulo $n_1$ and $n_2$ respectively where $\gcd(n_1,n_2)=1$. Suppose their respective monodromy groups are $N_1\rtimes C_k$ and $N_2\rtimes C_k$. Then there exists an algebraic $k$-gon $[c_0,\dots,c_{k-1}]$ modulo $n_1n_2$ with monodromy group $(N_1\times N_2)\rtimes C_k$. Furthermore, if $a_i\not\equiv 0\mod n_1$ or $b_i\not\equiv 0\mod n_2$ for every $i$, then $c_i\not\equiv 0\mod n_1n_2$ for all $i$.
\end{proposition}

\begin{proof}
By the Chinese Remainder Theorem, there exist unique integers $c_i$ with $0<c_i<n_1n_2$ such that $c_i\equiv a_i\mod n_1$ and $c_i\equiv b_i\mod n_2$ for all $i$. Since $c_i\equiv a_i\mod n_1$, we see that $c_0+\dots +c_{k-1}\equiv 0\mod n_1$ and $\gcd(c_0,\dots,c_{k-1},n_1)=1$. A similar argument shows that $c_0+\dots +c_{k-1}\equiv 0\mod n_2$ and $\gcd(c_0,\dots,c_{k-1},n_2)=1$. Hence, $c_0+\dots +c_{k-1}\equiv 0\mod n_1n_2$ and $\gcd(c_0,\dots,c_{k-1},n_1n_2)=1$ since $\gcd(n_1,n_2)=1$. 

Now, we will will compute the monodromy group of $[c_0,\dots,c_{k-1}]$ which is  $N\rtimes C_k$ where $N$ is an abelian group and submodule of $(\Z\slash n_1n_2\Z)^k$. Since $c_i\equiv a_i\mod n_1$ for all $i$ , we see that 
\begin{align}\label{equation:modular_monodromy_modn}
C'=\begin{bmatrix}
c_0&c_{k-1}&\dots& c_2 &c_1\\
c_1& c_0&c_{k-1}& &c_2\\
\vdots& c_1&c_0&\ddots & \vdots\\
c_{k-2}&  & \ddots &\ddots&c_{k-1}\\
c_{k-1}&c_{k-2}&\dots&c_1&c_0
\end{bmatrix}&\equiv 
\begin{bmatrix}
a_0&a_{k-1}&\dots& a_2 &a_1\\
a_1& a_0&a_{k-1}& &a_2\\
\vdots& a_1&a_0&\ddots & \vdots\\
a_{k-2}&  & \ddots &\ddots&a_{k-1}\\
a_{k-1}&a_{k-2}&\dots&a_1&a_0
\end{bmatrix}=C''\mod n_1.
\end{align}

Let $d_1,\dots,d_k$ be the elementary divisors of $C'$. They are the same modulo $n_1$ as the elementary divisors of $C''$. By Theorem \ref{theorem:main2}, we know the monodromy group of $[c_0,\dots,c_{k-1}]$ is 
\begin{equation}\label{equation:direct_product}
\bigoplus_{i=1}^k{C_{\delta_i}}=\bigoplus_{i=1}^k{C_{\frac{n_1n_2}{\gcd(d_i,n_1n_2)}}}=\bigoplus_{i=1}^{k}{C_{\frac{n_1}{\gcd(d_i,n_1)}}\oplus C_{\frac{n_2}{\gcd(d_i,n_2)}}}
\end{equation}
since $\gcd(n_1,n_2)=1$. Thus, the monodromy group of $[a_0,\dots,a_{k-1}]$ is $N_1=\displaystyle\bigoplus_{i=1}^k{C_{\frac{n_1}{\gcd(d_i,n_1)}}}$. Therefore, $N_1\cong n_2N\cong N\slash n_1N$. If $N_2$ is the monodromy group of $[b_0,\dots,b_{k-1}]$, then a similar argument shows that $N_2\cong n_1N\cong N\slash n_2N$. We conclude that $N\cong N_1\times N_2$ and the main result follows.
\end{proof}

\begin{remark}
In essence, Proposition \ref{proposition:combine_groups} allows one to combine two algebraic $k$-gons with coprime moduli and create a new algebraic $k$-gon $[c_0,\dots,c_{k-1}]$. The monodromy group of $[c_0,\dots,c_{k-1}]$ is a combination of the monodromy groups of the original two algebraic $k$-gons.
\end{remark}

\begin{example}
One can actually combine two algebraic $k$-gons with \underline{no} $k$-gon associates to create an algebraic $k$-gon with a $k$-gon associate. Consider the algebraic $3$-gon $[0,1,1]$ modulo $2$ with monodromy group $C_2^2\rtimes C_3$ and the algebraic $3$-gon $[1,0,4]$ modulo $5$ with monodromy group $C_5^2\rtimes C_3$. Neither of these algebraic $3$-gons have a polygonal associate. However, if we combine them using Proposition \ref{proposition:combine_groups}, we obtain the algebraic $3$-gon $[6,5,9]$ modulo $10$. This algebraic $3$-gon has a $3$-gon associate $[4,5,1]$ modulo $10$ obtained by scaling by $9\mod 10$. The $3$-gon $[4,5,1]$ has monodromy group $(C_2^2\times C_5^2)\rtimes C_3\cong C_{10}^2\rtimes C_3$.
\end{example}

\begin{proposition}\label{proposition:project_group}
Suppose that $[c_0,\dots,c_{k-1}]$ is an algebraic $k$-gon modulo $n_1n_2$ with $n_1,n_2>1$ and with monodromy group $N\rtimes C_k$. Then there exists an algebraic $k$-gon $[a_0,\dots,a_{k-1}]$ modulo $n_1$ with monodromy group $(n_2N)\rtimes C_k$. If $\gcd(n_1,n_2)=1$, then the monodromy group $(n_2N)\rtimes C_k\cong (N\slash n_1N)\rtimes C_k$.
\end{proposition}

\begin{proof}
Observe that $c_i\not\equiv 0\mod n_1$ for some $i$. If $n_1|c_i$ for all $i$, then $\gcd(c_0,\dots,c_{k-1},n_1n_2)>1$, a contradiction with the definition of an algebraic polygon.

Choose $a_i\equiv c_i\mod n_1$ for all $i$. We see that $a_0+\dots +a_{k-1}\equiv 0\mod n_1$ since $c_0+\dots +c_{k-1}\equiv 0\mod n_1n_2.$ Suppose that the monodromy group of $[a_0,\dots,a_{k-1}]$ is $N_1\rtimes C_k$. By the exact same calculation as in \eqref{equation:modular_monodromy_modn}, we conclude that $N_1\cong n_2N$.

Now suppose $\gcd(n_1,n_2)=1$. Using \eqref{equation:direct_product}, we see that the monodromy group of $[c_0,\dots,c_{k-1}]$ has the form $N\rtimes C_k$ where 
\begin{equation*}
N=\bigoplus_{i=1}^k{C_{\delta_i}}=\bigoplus_{i=1}^{k}{C_{\frac{n_1}{\gcd(d_i,n_1)}}\oplus C_{\frac{n_2}{\gcd(d_i,n_2)}}}
\end{equation*}
Observe that $n_2N\cong \displaystyle\bigoplus_{i=1}^{k}{C_{\frac{n_1}{\gcd(d_i,n_1)}}}$ and $n_1N\cong \displaystyle\bigoplus_{i=1}^{k}{C_{\frac{n_2}{\gcd(d_i,n_2)}}}$. Thus, $n_2N\cong N\slash n_1N.$
\end{proof}

\begin{remark}
If $n_1$ and $n_2$ are coprime in Proposition \ref{proposition:project_group}, then $N_1\cong N\slash n_1N$. However, this is not the case when $n_1$ and $n_2$ have a non-trivial gcd. We illustrate this phenomenon in the following example.
\end{remark}

\begin{example}
Consider the $k$-gon $[1,2,24,23]$ modulo $25$. The monodromy group is $N\rtimes C_4$ where $N\cong C_{25}\times C_5$. If we apply Proposition \ref{proposition:project_group} when $n_1=5$, we obtain the $k$-gon $[1,2,4,3]$ which has monodromy group $N_1\rtimes C_4$ where $N_1\cong C_5\cong 5N\not\cong N\slash 5N$.
\end{example}

The following proposition allows us to lift an algebraic $k$-gon modulo $n$ to an algebraic $\ell$-gon modulo $n$ if $k|\ell$.

\begin{proposition}\label{proposition:raising_rank_modular}
Suppose that $k,\ell\in\N$ and $k|\ell$. Further suppose that $[a_0,\dots,a_{k-1}]$ is an algebraic $k$-gon modulo $n$ with monodromy group $N\rtimes C_k$. Then there exists an algebraic $\ell$-gon $[c_0,\dots,c_{\ell-1}]$ modulo $n$ with monodromy group $N\rtimes C_\ell$.
\end{proposition}

\begin{proof}
Let $c_i=a_j$ where $j$ is the least nonnegative integer satisfying $i\equiv j\mod k$. In essence, $$[c_0,\dots,c_{\ell-1}]=[a_0,\dots,a_{k-1},a_0,\dots, a_{k-1},a_0,\dots,a_{k-1}]$$
where the pattern $a_0,\dots,a_{k-1}$ repeats itself $\frac{\ell}{k}$ times. Let 
$$
C=
\begin{bmatrix}
a_0&a_{k-1}&\dots& a_2 &a_1\\
a_1& a_0&a_{k-1}& &a_2\\
\vdots& a_1&a_0&\ddots & \vdots\\
a_{k-2}&  & \ddots &\ddots&a_{k-1}\\
a_{k-1}&a_{k-2}&\dots&a_1&a_0
\end{bmatrix}
$$
and observe that 
$$
C'=\begin{bmatrix}
c_0&c_{\ell-1}&\dots& c_2 &c_1\\
c_1& c_0&c_{\ell-1}& &c_2\\
\vdots& c_1&c_0&\ddots & \vdots\\
c_{\ell-2}&  & \ddots &\ddots&c_{\ell-1}\\
c_{\ell-1}&c_{\ell-2}&\dots&c_1&c_0
\end{bmatrix}
=
\begin{bmatrix}
C & C&\dots &C & C\\
C & C & C & &C\\
\vdots & C & C&\ddots & \vdots\\
C& & \vdots & \vdots & C\\
C& C& \dots & C &C
\end{bmatrix}
$$
where the matrix $C$ appears $\frac{\ell}{k}$ times in each row and column. Therefore, the group generated by the columns of $C'$ is isomorphic to the group generated by the columns of $C$ and thus the monodromy group of $[c_0,\dots,c_{\ell-1}]$ is $N\rtimes C_\ell$.
\end{proof}

The following example illustrates how Proposition \ref{proposition:raising_rank_modular} is used to lift an algebraic $k$-gon to an algebraic $\ell$-gon.

\begin{example}
Let $k=2$, $\ell=4$ and consider the algebraic $2$-gon $[3,4]$ modulo $n=7$. Using Proposition \ref{proposition:raising_rank_modular}, lift $[3,4]$ to the algebraic $4$-gon $[3,4,3,4]$ modulo $7$. The monodromy group of $[3,4]$ is $C_7\rtimes C_2$ and the monodromy group of $[3,4,3,4]$ is $C_7\rtimes C_4$.
\end{example}

A quick lemma about semidirect products is needed to complete our series of results about combining algebraic polygons to form new algebraic polygons.

\begin{lemma}\label{lemma:semidirect_product}
Suppose that $N_1,H_1,N_2,H_2$ are finite groups. If $G_1\cong N_1\rtimes H_1$ and $G_2\cong N_2\rtimes H_2$ then $G_1\times G_2\cong (N_1\times N_2)\rtimes (H_1\times H_2)$.
\end{lemma}

\begin{proof}
An element of the group $G_1\times G_2$ has the form $((n_1,h_1),(n_2,h_2))$ where $n_1\in N_1$, $n_2\in N_2$, $h_1\in H_1$, and $h_2\in H_2$. Let $\epsilon$ represent the identify in the respective group. Consider the subgroups 
$$
N=\langle ((n_1,\eps),(\eps,\eps)), ((\eps,\eps),(n_2,\eps)):n_1\in N_1,\ n_2\in N_2\rangle
$$
and
$$
H=\langle ((\eps,h_1),(\eps,\eps)),((\eps,\eps),(\eps, h_2)):h_1\in H_1,\ h_2\in H_2 \rangle.
$$
It is easy to see that $N\cong N_1\times N_2$ and $H\cong H_1\times H_2$ and $NH=G_1\times G_2$. It is also easy to see that $N\cap H$ contains only the identity of $G_1\times G_2$. In order to prove that $G_1\times G_2$ is isomorphic to $N\rtimes H$, we need to prove that $N\vartriangleleft G_1\times G_2$. This follows immediately from the fact that $N_1\vartriangleleft G_1$ and $N_2\vartriangleleft G_2$.
\end{proof}

Now, let us combine the results from Propositions \ref{proposition:combine_groups} and \ref{proposition:raising_rank_modular} to obtain the following corollary.

\begin{corollary}\label{corollary:monodromy_product}
Fix $n_1,n_2,k,\ell\in \N$ with $k,\ell\ge 2$. Suppose that $\gcd(n_1,n_2)=1$ and $\gcd(k,\ell)=1$. If $[a_0,\dots,a_{k-1}]$ is an algebraic $k$-gon modulo $n_1$ with monodromy group $N_1\rtimes C_k$ and $[b_0,\dots,b_{\ell-1}]$ is an algebraic $\ell$-gon modulo $n_2$ with monodromy group $N_2\rtimes C_{\ell}$, then there exists an algebraic $k\ell$-gon $[c_0,\dots,c_{k\ell-1}]$ modulo $n_1n_2$ with monodromy group $(N_1\times N_2)\rtimes C_{\ell k}\cong (N_1\rtimes C_k)\times (N_2\rtimes C_\ell)$.
\end{corollary}

\begin{proof}
Combining Propositions \ref{proposition:combine_groups} and \ref{proposition:raising_rank_modular} give us the desired algebraic $k\ell$-gon $[c_0,\dots,c_{k\ell-1}]$ with monodromy group $(N_1\rtimes N_2)\rtimes C_{k\ell}$. Since $\gcd(k,\ell)=1$, $C_{k\ell}\cong C_k\times C_\ell$. Thus, by Lemma \ref{lemma:semidirect_product}, $(N_1\rtimes N_2)\rtimes C_{k\ell}\cong (N_1\rtimes C_k)\times (N_2\rtimes C_\ell)$.
\end{proof}

The following example illustrates how to use Corollary \ref{corollary:monodromy_product}.

\begin{example}
Let $k=3$, $\ell=4$, $n_1=7$ and $n_2=5$. Let $[1,2,4]$ be our algebraic $3$-gon modulo $7$ and let $[2,3,3,2]$ be our algebraic $4$-gon modulo $5$. The monodromy group of group of $[1,2,4]$ is $C_7\rtimes C_3$ and the monodromy group of $[2,3,3,2]$ is $C_5^2\rtimes C_4$. Using Proposition \ref{proposition:raising_rank_modular}, we lift $[1,2,4]$ to $[1,2,4,1,2,4,1,2,4,1,2,4]$ and we lift $[2,3,3,2]$ to $[2,3,3,2,2,3,3,2,2,3,3,2]$. Using Proposition \ref{proposition:combine_groups}, we combine these algebraic $12$-gons to obtain $[22,23,8,22,2,18,8,2,32,8,23,32]$ modulo $35$ which has monodromy group $(C_7\times C_5^2)\rtimes C_{12}\cong (C_7\rtimes C_3)\times (C_5^2\rtimes C_4)$.
\end{example}

\section{Results about Circulant Matrices}\label{section:circ_matrices}

The following results on circulant matrices will be needed to compute monodromy groups of polygons modulo $p$ when $p$ is prime. The results are well known over $\C$, and we provide the proofs for the corresponding results over finite fields for completeness.

\begin{definition}
A $k\times k$ {\it circulant} matrix $C$ has the following form
$$
C=\begin{bmatrix}
a_0&a_{k-1}&\dots& a_2 &a_1\\
a_1& a_0&a_{k-1}& &a_2\\
\vdots& a_1&a_0&\ddots & \vdots\\
a_{k-2}&  & \ddots &\ddots&a_{k-1}\\
a_{k-1}&a_{k-2}&\dots&a_1&a_0
\end{bmatrix}.
$$
\end{definition}

For the purposes of this paper, the entries $c_i$ are integers or integers modulo $n$. 

\begin{definition}
We call the polynomial $f(x)=a_0+a_1x+\dots+a_{k-1}x^{k-1}$ the {\it associated polynomial} of the circulant matrix $C$.
\end{definition}

\begin{lemma}\label{lemma:circulant_eigen}
Assume that you have a $k\times k$ circulant matrix $C$ with entries in a field $\F$. Also assume that $\omega$ is a primitive $k$th root of unity. The vectors $$v_j=\begin{bmatrix} 1\\\omega^j\\ \omega^{2j}\\ \vdots \\ \omega^{(k-1)j}\end{bmatrix}$$ for $j=0,\dots,k-1$ are eigenvectors of a circulant matrix $C$ with respective eigenvalues $\lambda_j=a_0+a_{k-1}\omega^j+a_{k-2}\omega^{2j}+\dots +a_1\omega^{(k-1)j}$.
\end{lemma}

\begin{proof}
Let $(C\cdot v_j)_i$ denote the $i$th entry of the column vector $C\cdot v_j$. Observe that 
\begin{align*}
    (C\cdot v_j)_i &= \sum_{n=0}^ia_{i-n}\cdot \omega^{nj} + \sum_{n=0}^{k-i-2}a_{k-1-n}\cdot \omega^{(i+1+n)j}\\
    &=\omega^{ij}\left[ \sum_{n=0}^ia_{i-n}\cdot \omega^{(n-i)j} + \sum_{n=0}^{k-i-2}a_{k-1-n}\cdot \omega^{(1+n)j}\right]\\
    &=\omega^{ij}\left[ \sum_{n=0}^ia_{i-n}\cdot \omega^{(k+n-i)j} + \sum_{n=0}^{k-i-2}a_{k-1-n}\cdot \omega^{(1+n)j}\right].
\end{align*}

If we re-index the two sums letting $n'=i-n$ in the first sum and $n'=k-1-n$ in the second sum, we obtain:
\begin{align*}
(C\cdot v_j)_i &= \omega^{ij}\left[\sum_{n'=0}^i{a_{n'}\omega^{(k-n')j}}+\sum_{n'=i+1}^{k-1}{a_{n'}\omega^{j(k-n')}} \right]\\
&=\omega^{ij}\left[ a_0+a_{k-1}\omega^j+a_{k-2}\omega^{2j}+\dots +a_1\omega^{(k-1)j}\right]=\lambda_j\omega^{ij}.
\end{align*}

It follows that $C\cdot v_j = \lambda_j\cdot v_j$, so $v_j$ is an eigenvector with eigenvalue $\lambda_j$.
\end{proof}

\begin{lemma}
Assume that you have a $k\times k$ circulant matrix $C$ with entries in a field $\F$ that has $k$ distinct $k$th roots of unity. Then the eigenvectors $$v_j=\begin{bmatrix} 1\\\omega^j\\ \omega^{2j}\\ \vdots \\ \omega^{(k-1)j}\end{bmatrix}$$ form a basis for the vector space $\F^k$ and thus $C$ is diagonalizable.
\end{lemma}
\begin{proof}

Let $V$ be the matrix $[v_j:0\leq j \leq k-1]$. Note that $V$ is a Vandermonde matrix. Thus 
$$
\det(V) = \prod_{0\le i<j\le k-1}(\omega^j-\omega^i).
$$
Suppose $\det(V) = 0$. Then $\omega^j - \omega^i = 0$ for some $i\ne j$. This is a contradiction with the fact that $\F$ has $k$ distinct $k$th roots of unity. Hence, $\det(V)\ne 0$ and the eigenvectors $v_j$ form a basis of $\F^k$ and $C$ is diagonalizable.

\end{proof}

\begin{corollary}
If $C$ is a $k\times k$ circulant matrix over a field $\F$ which has an algebraic extension with $k$ distinct $k$th roots of unity, 
$$
\det(C)=\prod_{j=0}^{k-1}{(a_0+a_1\omega^j+\dots+a_{k-1}\omega^{(k-1)j})}=\prod_{j=0}^{k-1}{f(\omega^j)}
$$
where $f$ is the associated polynomial of $C$.
\end{corollary}
\begin{proof}
The determinant of a diagonalizable matrix is the product of the eigenvalues (including multiplicities) listed in Lemma \ref{lemma:circulant_eigen}.
\end{proof}

\begin{lemma}\label{lemma:circulant_rank}
The rank of a $k\times k$ circulant matrix $C$ over a field $\F$ which has an algebraic extension with $k$ distinct $k$th roots of unity is equal to $k-d$ where $d$ is the degree of $\gcd(f(x),x^k-1)$.
\end{lemma}
\begin{proof}

Let $d$ be the dimension of the null space of $C$ which is equal to the multiplicity of the eigenvalue $0$. An eigenvalue $\lambda_j=0$ if and only if $a_0+a_{k-1}\omega^j+\dots+a_1\omega^{(k-1)j}=f(\omega^j)=0$. Hence, the dimension of the null space is equal to the number of $k$th roots of unity which are also roots of $f(x)$. Therefore, $d$ is the degree of the polynomial $\gcd(f(x),x^k-1)$ and we obtain $\rank(C)=k-d$.
\end{proof}

\begin{lemma}\label{lemma:circulant_irreducible}
Suppose $p_1$ and $p_2$ are distinct prime integers and $p_1$ is a generator for the cyclic group $\F_{p_2}^\times$. Then $x^{p_2-1}+\dots+x+1$ is irreducible over $\F_{p_1}$.
\end{lemma}
\begin{proof}
Let $\omega$ be a primitive $p_2$th root of unity of $\F_{p_1}$. The group $\mathrm{Gal}(\F_{p_1}(\omega)\slash\F_{p_1})$ is generated by the Frobenius automorphism $\phi:x\mapsto x^{p_1}$ \cite[Proposition 5.8, page 445]{A09}. Since $p_1$ generates $\F_{p_2}^\times$, we see that $|\phi|=p_2-1$. Thus, $[\F_{p_1}(\omega):\F_{p_1}]=p_2-1$ and $x^{p_2-1}+\dots+x+1$ is irreducible over $\F_{p_1}$.
\end{proof}

\begin{corollary}\label{corollary:pq_irreducible}
Let $p_1$ and $p_2$ be primes such that $p_1$ is a generator for the cyclic group $\F_{p_2}^\times$. Suppose that $C$ is a $p_2\times p_2$ circulant matrix with entries in $\F_{p_1}$. Then $\rank(C)=0$, $1$, $p_2-1$, or $p_2$.
\end{corollary}
\begin{proof}
By Lemma \ref{lemma:circulant_irreducible}, we know that $x^{p_2-1}+\dots+x+1$ is irreducible over $\F_{p_1}$. Hence $x^{p_2}-1$ factors as $(x-1)(x^{p_2-1}+\dots+x+1)$ over $\F_{p_1}$. By Lemma \ref{lemma:circulant_rank}, we see that $d=0$, $1$, $p_2-1$, or $p_2$ from which our result follows.
\end{proof}

\section{Results for $n=p$ Prime}\label{section:prime_results}
In Section \ref{section:snf}, we gave a description of the monodromy group in terms of the elementary divisors of a particular circulant matrix. Although this result (Theorem \ref{theorem:main2}) allows one to easily compute the monodromy group, the result is not explicit. We will prove several results below in the special case when $n$ is equal to a prime $p$. In other words, $[a_0,\dots,a_{k-1}]$ represent an algebraic $k$-gon modulo a prime $p$. In this case, the group $N$ can be viewed as a $\Z\slash p\Z=\Fp$ module  and is thus a vector space. In this section, $\Fp$ will denote the finite field with $p$ elements and $\Fp^\times$ will denote its group of units.


\begin{proposition}\label{proposition:main3}
Suppose that $[a_0,\dots,a_{k-1}]$ represents an algebraic $k$-gon modulo a prime $p$. Let $f(x)=a_0+a_1 x+\dots+a_{k-1}x^{k-1}$ and let $d$ be the degree of $\gcd(f(x),x^k-1)$. Then the monodromy group of $[a_0,\dots,a_{k-1}]$ is $G(a_0,\ldots,a_{k-1})=C_p^{k-d} \rtimes C_k$.
\end{proposition}

\begin{proof}
By Lemma \ref{lemma:circulant_rank}, we know that the rank of the matrix 
$$
C=\begin{bmatrix}
a_0&a_{k-1}&\dots& a_2 &a_1\\
a_1& a_0&a_{k-1}& &a_2\\
\vdots& a_1&a_0&\ddots & \vdots\\
a_{k-2}&  & \ddots &\ddots&a_{k-1}\\
a_{k-1}&a_{k-2}&\dots&a_1&a_0
\end{bmatrix}
$$
is equal to $k-d$ where $d$ is the degree of $\gcd(f(x),x^k-1)$. The rank of a subspace of a vector space determines the group structure and the result follows.
\end{proof}

This allows us to translate the problem of finding the rank of a matrix to that of a degree of a gcd. The following corollary shows how we can use this connection to compute the monodromy groups of a large collection of dessins on rational billiards surfaces.

\begin{corollary}\label{corollary:full_monodromy}
Suppose $p_2$ is a prime number and $p_1$ is a prime number that generates the cyclic group $(\F_{p_2})^\times$. Suppose that $[a_0,\dots,a_{p_2-1}]$ represents an algebraic $p_2$-gon modulo $p_1$ with monodromy group $G(a_0,\ldots,a_{p_2-1})$. Let $f(x)=a_0+a_1 x+\dots+a_{p_2-1}x^{p_2-1}$, then $G(a_0,\ldots,a_{p_2-1})\cong C_{p_1}^{p_2-1} \rtimes C_{p_2}$.
\end{corollary}

\begin{proof}
By Corollary \ref{corollary:pq_irreducible}, the rank of the appropriate matrix $C$ is $0$, $1$, $p_2-1$, or $p_2$. Since $f(1)\equiv 0\mod p_1$, we know $x-1|f(x)$ and thus $\rank(C)\le p_2-1$. Since $x^{p_2-1}+\dots+x+1$ is irreducible over $\F_{p_1}$ by Lemma \ref{lemma:circulant_irreducible}, the $\deg(\gcd(f(x),x^{p_2}-1))=1$ or $p_2$. If $\deg(\gcd(f(x),x^{p_2}-1))=p_2$ then $a_0=\dots=a_{p_2-1}=0$ since $\deg(f)\le p_2-1$, which is a contradiction. Hence, $\deg(\gcd(f(x),x^{p_2}-1))=1$ and the result follows.
\end{proof}

\begin{example}
Choose $p_2=17$. Observe that $p_1=41$ generates the multiplicative group $\F_{17}^\times$. Hence, any algebraic $17$-gon modulo $41$ has monodromy group $C_{41}^{16}\rtimes C_{17}$.
\end{example}

\subsection{Possible Monodromy Groups}

Now, let's prove a general theorem that lists all possible monodromy groups for polygons $[a_0,\dots,a_{k-1}]$ modulo $p$.

\begin{proposition}\label{proposition:possible_groups}
Suppose that $[a_0,\dots,a_{k-1}]$ represents an algebraic polygon modulo a prime $p$. Let $f(x)=a_0+a_1 x+\dots+a_{k-1}x^{k-1}$ and suppose $x^k-1=\prod{g_i(x)}$ where the $g_i(x)$ are irreducible over $\Fp$. Further suppose that $\gcd(f(x),x^k-1)=\displaystyle\prod_{j=1}^{\ell}{g_{i_j}(x)}$. Then the monodromy group of $[a_0,\dots,a_{k-1}]$ is $G(a_0,\ldots,a_{k-1})=C_p^{k-d} \rtimes C_k$ where $d=\displaystyle\sum_{j=1}^{m}{\deg(g_{i_j}(x))}$.

\end{proposition}

In essence, Proposition \ref{proposition:possible_groups} gives a list of all potential monodromy groups of \emph{algebraic} $k$-gons modulo $p$. If $[a_0,\dots,a_{k-1}]$ is an algebraic $k$-gon modulo $p$ with monodromy group $C_p^{k-d}\rtimes C_k$ then $d$ must be equal to the sum of degrees of distinct irreducible factors of $x^k-1$ in $\Fp$. The factor $x-1$ must be one of these factors. If there is no way to add up to $d$ the degrees $\deg(g_{i_j}(x))$ of a subset of the irreducible factors $g_i(x)$ of $x^k-1$ in $\Fp$, then such a monodromy group \emph{cannot} occur.

\begin{example}
Consider $k=3$ and $p=5$. We see that $x^3-1$ factors as $(x-1)(x^2+x+1)$ modulo $5$. Since $x-1$ is required to be a factor of $\gcd(f(x),x^k-1)$, we see that this gcd \emph{cannot} have degree two. Therefore, the monodromy group $C_5^{3-2}\rtimes C_3$ is not achieved by any algebraic $3$-gon modulo $5$.
\end{example}

\begin{proof}[Proof of Proposition \ref{proposition:possible_groups}]
By Proposition \ref{proposition:main3}, we know that $d$ is the degree of $\gcd(f(x),x^k-1)$. Since the gcd must be a product of some subset of $\{g_i(x)\}$, we see that $d$ is the sum of the degrees of some subset of $\{g_i(x)\}$. The theorem follows.

Observe that $\ell\ge 1$ because $f(1)=a_0+\dots a_{k-1}\equiv 0\mod p$ implies $x-1$ divides $f(x)$. 
\end{proof}

\begin{theorem}\label{theorem:algebraic_monodromy_groups}
Fix a prime $p\nmid k$. Suppose $x^k-1=\prod{g_i(x)}$ where the $g_i(x)$ are irreducible over $\Fp$. Let $d=\displaystyle\sum_{j=1}^{\ell}{\deg(g_{i_j}(x))}$. Further suppose that $g_{i_j}=x-1$ for some $i_j$. Then there exists an \emph{algebraic} $k$-gon $[a_0,\dots,a_{k-1}]$ modulo $p$ with monodromy group $G(a_0,\ldots,a_{k-1})\cong C_p^{k-d} \rtimes C_k$.
\end{theorem}

\begin{proof}
Let $f(x)=\displaystyle\prod{g_{i_j}(x)}$. We see that $\deg(\gcd(f(x),x^k-1))=d$. If $f(x)=a_0+\dots +a_{k-1}x^{k-1}$ then $a_0+\dots+a_{k-1}\equiv 0\mod p$ since $(x-1)|f(x)$. Since $f(x)$ is not the zero polynomial over $\Fp$, we see that $\gcd(a_0,\dots,a_{k-1},p)=1$. Therefore, $[a_0,\dots,a_{k-1}]$ is an algebraic $k$-gon with monodromy group $G(a_0,\ldots,a_{k-1})\cong C_p^{k-d} \rtimes C_k$ by Proposition \ref{proposition:main3}.
\end{proof}

Theorem \ref{theorem:algebraic_monodromy_groups} proves that all possible monodromy groups from Proposition \ref{proposition:possible_groups} are achieved by algebraic polygons modulo $p$ for a fixed prime $p$. Therefore, it is natural to ask which groups can occur for \emph{$k$-gons} modulo $p$. The following theorem shows that for primes $p>k$, all possible monodromy groups from Proposition \ref{proposition:possible_groups} are achieved by {\it $k$-gons} modulo $p$.

\begin{theorem}\label{theorem:monodromy_groups}
Fix a prime $p>k\ge 3$. Suppose $x^k-1=g_1(x)\cdots g_\ell(x)$ where the $g_i(x)$ are irreducible over $\Fp$. Let $d=\displaystyle\sum_{j=1}^{m}{\deg(g_{i_j}(x))}$ where $m$ is a positive integer less than $\ell$ and $1\le i_1<\dots<i_m\le \ell$. Further suppose that $g_{i_j}=x-1$ for some $i_j$. Then there exists a $k$-gon $[a_0,\dots,a_{k-1}]$ modulo $p$ with monodromy group $G(a_0,\ldots,a_{k-1})\cong C_p^{k-d} \rtimes C_k$.
\end{theorem}

\begin{remark}
We only consider primes $p>k$ in Theorem \ref{theorem:monodromy_groups}, because $p\nmid k$ in this case. The polynomial, $x^k-1$, has no repeated factors over $\Fp$ when $p\nmid k$ which implies that there is an algebraic extension of $\Fp$ with $k$ distinct $k$th roots of unity. Furthermore, Theorem \ref{theorem:monodromy_groups} is not true for primes $p\le k$ in its current formulation. Consider $k=3$ and $p=3$. Since $x^3-1=(x-1)^3$ modulo $3$, Theorem \ref{theorem:monodromy_groups} would predict the existence of $3$-gons with monodromy groups $C_3^2\rtimes C_3$ and $C_3\rtimes C_3$. However, the only $3$-gon is $[1,1,1]$, and thus the only possible monodromy group of a $3$-gon modulo $3$ is $C_3\rtimes C_3$.
\end{remark}

\section{Proving Theorem \ref{theorem:monodromy_groups}}
\label{section:proving_monodromy_groups}
Here we lay out the basic strategy and supporting lemmas we will use to prove Theorem \ref{theorem:monodromy_groups}.

\subsection{Strategy for Proving Theorem \ref{theorem:monodromy_groups}}

Recall that Lemmas 8 and 9 allow us to construct a geometric polygon with monodromy group $G$ if we can find an algebraic polygon with all nonzero entries that has an isomorphic monodromy group. To control the number of nonzero entries in an algebraic polygon, we define:

\begin{definition}
For a polynomial $f(x)=a_nx^n+\dots+a_1x+a_0$ with $a_n\ne 0$, let $w(f(x))$ be the maximum number of consecutive coefficients of $f(x)$ that are zero. For example, if $g(x)=x^7-x^3+1$, then $w(g(x))=3$ since $a_6=a_5=a_4=0$ while $a_3,a_7\ne 0$.
\end{definition}

Now, our strategy for proving Theorem \ref{theorem:monodromy_groups} is the following:
\begin{enumerate}
    \item For a given monodromy group $G\cong C_p^{k-d}\rtimes C_k$ described in Theorem \ref{theorem:monodromy_groups}, find an appropriate polynomial $g(x)$ satisfying $g(x)|x^k-1$, $x-1|g(x)$, and $\deg(g(x))=d$. 
    \item Using Proposition \ref{proposition:close_zero_gap}, multiply $g(x)$ by a series of linear polynomials to produce a polynomial $f(x)$, each of which reduces the value of the $w$ function but leaves $\gcd(f(x),x^k-1)=g(x)$. Repeat until $g(x)$ has been transformed into a polynomial $f(x)=\sum b_ix^i$ of degree $k-1$ with $w(f(x))=0$ and $\gcd(f(x),x^k-1)=g(x)$.
    \item Use Lemmas \ref{lemma:associate} and \ref{lemma:associate_groups} to transform $[b_0,\ldots,b_{k-1}]$ into a geometric polygon with monodromy group $G$.
\end{enumerate}

\begin{remark}
    The proofs of Theorem \ref{theorem:rank} and Proposition \ref{proposition:small_d} follow the above approach. However, the proof of Proposition \ref{proposition:large_d} differs slightly.
\end{remark}

In the following proposition, we show that if we choose $\alpha$ appropriately, then $w((x-\alpha)\cdot f(x))=\max(w(f(x))-1,0)$.

\begin{proposition}\label{proposition:close_zero_gap}
Let $\F$ be a field. Suppose that $f(x)=a_nx^n+a_{n-1}x^{n-1}+\dots +a_1 x+a_0\in\F[x]$ with $a_0,a_n\ne 0$. If $\{\alpha_0,\dots,\alpha_n\}$ are distinct non-zero elements of $\F$, then there exists at least one $\alpha_i$ such that $w(f(x)\cdot (x-\alpha_i))=\max(w(f(x))-1,0)$.
\end{proposition}
\begin{proof}
Consider the coefficients of $f(x)\cdot (x-\alpha_i)=b_{n+1}x^{n+1}+b_nx^n+\dots+b_1 x+b_0$. Observe that $b_0,b_{n+1}\ne 0$. Further observe that for $0<j<n+1$, $b_j=a_{j-1}-\alpha_i a_j$. If $b_j=0$ then one of three situations must arise:
\begin{enumerate}[(a)]
    \item $a_{j-1}=a_j=0$
    \item $a_{j-1}=\alpha_i=0$
    \item $\alpha_i=\frac{a_{j-1}}{a_j}$ and $a_j\ne 0$
\end{enumerate}
Situation (b) cannot arise, because $\alpha_i$ is chosen from non-zero elements of $\F$. By the pigeon hole principle, there exists at least one $\alpha_i$ in $\{\alpha_0,\dots,\alpha_n\}$ such that $\alpha_i\ne \frac{a_{j-1}}{a_j}$ for all $0\le j\le n$. Our choice of $\alpha_i$ prevents situation (c) from arising. Since situation (a) cannot occur if $w(f(x))=0$ then $w(f(x)\cdot (x-\alpha_i))=0$ in this case.

Now we consider the case where $w(f(x))>0$. By our choice of $\alpha_i$, $b_j=0$ implies $a_j=a_{j-1}=0$. Assume that $w(f(x))=d+1$ which implies there exist $a_\ell,\dots,a_{\ell+d}$ with $a_{\ell-1}\ne 0$ and $a_{\ell+d+1}\ne 0$. We see that $b_\ell\ne 0$ and $b_{\ell+d+1}\ne 0$ and $b_{\ell+1},\dots,b_{\ell+d}=0$. Hence, we have shown that $w(f(x)\cdot (x-\alpha_i))=w(f(x))-1$.
\end{proof}

Now, we prove a useful result about the gcd of collections of polynomials with $x^k-1$.

\begin{lemma}\label{lemma:rotate}
Let $\F$ be a field and let $f(x)=a_{d}x^{d}+\dots+a_1x+a_0\in \F[x]$. Then $\gcd(f(x),x^k-1)=\gcd(x\cdot f(x)-a_{k-1}(x^k-1),x^k-1)$. 
\end{lemma} 

\begin{proof}
Observe that $\gcd(f(x),x^k-1)=\gcd(x\cdot f(x),x^k-1)$ since $x$ does not divide $x^k-1$. If $g(x)=\gcd(x\cdot f(x),x^k-1)$ then it is clear that $g(x)$ divides $x\cdot f(x)-a_{k-1}(x^k-1)$. If $h(x)=\gcd(x\cdot f(x)-a_{k-1}(x^k-1),x^k-1)$, then $h(x)$ divides $x\cdot f(x)-a_{k-1}(x^k-1)+a_{k-1}(x^k-1)=x\cdot f(x)$. Hence, $\gcd(f(x),x^k-1)=\gcd(x\cdot f(x)-a_{k-1}(x^k-1),x^k-1)$.
\end{proof}

\begin{example}
Consider the field $\F_{2}$. Let $f(x)=x^5+x^2+x+1$. Using Lemma \ref{lemma:rotate}, we deduce that 
$$
\gcd(x^5+x^2+x+1,x^7-1)=\gcd(x^6+x^3+x^2+x, x^7-1)=\gcd(x^4+x^3+x^2+1,x^7-1)
$$
\end{example}

In the following proposition, we prove a result about the maximum value of $w(f(x))$ for polynomials $f(x)$ dividing $x^k-1$.

\begin{proposition}\label{proposition:zero_spread}
Let $\F$ be a field and let $f(x)=a_dx^d+\dots a_1 x+a_0\in \F[x]$. Suppose that $f(x)$ is a non-zero polynomial with $\deg(f(x))=d$ and $f(x)|x^k-1$. Then $w(f(x))<k-d$.
\end{proposition}

\begin{proof}
If $k-d\ge d=\deg(f(x))$, then the result is trivial. Otherwise, suppose that $w(f(x))\ge k-d$. This implies that $f(x)$ has $k-d$ consecutive coefficients equal to zero. For the purposes of this proof, assume $a_{\ell},\dots ,a_{\ell+k-d-1}=0$ for some $\ell<d$. 

Use Lemma \ref{lemma:rotate} exactly $k-(\ell+k-d-1)-1=d-\ell$ times. That is, consider
$$
g(x)=x^{d-\ell} f(x)-(a_{k-1}x^{d-\ell-1}+a_{k-2}x^{d-\ell-2}+\dots+x\cdot a_{\ell+k-d+1}+a_{\ell+k-d}  )(x^k-1)
$$
which can be rewritten as
$$
g(x)=\sum_{i=d}^{k-1}{a_{i-d+\ell} x^i}+\sum_{i=d-\ell}^{d-1}{a_{i-d+\ell} x^i}+\sum_{i=0}^{d-\ell-1}{a_{i+k-d+\ell}x^i}.
$$

Lemma \ref{lemma:rotate} states that $\gcd(g(x),x^k-1)=\gcd(f(x),x^k-1)$. Since the first summation above is equal to zero, we see that $\deg(g(x))\le d-1$. This implies that $\deg(\gcd(g(x),x^k-1))<d$ which is a contradiction since $\gcd(f(x),x^k-1)=f(x)$ and $\deg(f(x))=d$.
\end{proof}

This proposition allows us to immediately obtain the following interesting corollary. Though the result is surely known, we could not find a reference for it.

\begin{corollary}
If $f(x)=a_{k-1}x^{k-1}+a_{k-2}x^{k-2}+\dots +a_1x+a_0$ divides $x^k-1$ in a field $\F$ and $\deg(f(x))=k-1$ then $a_i\ne 0$ for $0\le i\le k-1$.
\end{corollary}

\subsection{Proving Theorem \ref{theorem:monodromy_groups} for $p>k+1$}
In this section, we prove Theorem \ref{theorem:monodromy_groups} in the case where $p>k+1$.

\begin{proposition}\label{proposition:reducing_rank}
Fix an integer $k\ge 3$. Suppose that $d|k$ and $d<k$. For primes $p>k$, there exists a $k$-gon $[a_0,\dots,a_{k-1}]$ modulo $p$ with monodromy group $G(a_0,\ldots,a_{k-1})\cong C_p ^{k-d} \rtimes C_k$.
\end{proposition}

\begin{proof}
Consider $f(x)=(x^d-1)^{\frac{k}{d}-1}(x^{d-1}+\cdots+x+1)=b_0+b_1x+\dots+b_{k-1}x^{k-1}$. Since $p>k$, the binomial coefficients in the expansion of $(x^d-1)^{\frac{k}{d}-1}$ are nonzero, and thus $b_i\not\equiv 0\mod p$ for $0\le i\le k-1$. Further observe that $x^k-1$ has no repeated factors since $p\nmid k$. Since $x^{d-1}+\dots +x+1$ divides $x^d-1$ and $x^d-1$ divides $x^k-1$, we deduce that $\gcd(f(x),x^k-1)=x^d-1$. Therefore, $[b_0,\dots,b_{k-1}]$ is an algebraic $k$-gon modulo $p$.

By Lemma \ref{lemma:associate}, Lemma \ref{lemma:associate_groups}, Proposition \ref{proposition:adjust_angles}, and Proposition \ref{proposition:main3}, $[b_0,\dots,b_{k-1}]$ has a $k$-gon associate $[a_0,\dots,a_{k-1}]$ modulo $p$ with monodromy group $G(a_0,\ldots,a_{k-1})\cong C_p^{k-d}\rtimes C_k$.
\end{proof}


The following theorem is crucial in the proof of Theorem \ref{theorem:monodromy_groups}.

\begin{theorem}\label{theorem:rank}
Let $k\ge 3$ be an integer, and let $p>k$ be a prime. Suppose $x^k-1=\prod{g_i(x)}$ where the $g_i(x)$ are irreducible over $\Fp$. Let $d=\displaystyle\sum_{j=1}^{\ell}{\deg(g_{i_j}(x))}$. Let $M$ equal the number of roots of $\frac{x^k-1}{\prod{g_{i_j}}}$ in $\Fp$. Further suppose that $g_{i_j}=x-1$ for some $i_j$. If $p>k+M$, there exists a $k$-gon $[a_0,\dots,a_{k-1}]$ modulo $p$ with monodromy group $G(a_0,\ldots,a_{k-1})=C_p^{k-d} \rtimes C_k$.
\end{theorem}

\begin{proof}
Let $g(x)=\prod{g_{i_j}(x)}$ which implies $\deg(g(x))=d$. By Proposition \ref{proposition:zero_spread}, $w(g(x))<k-d$. To produce a degree $k-1$ polynomial $f(x)$ with $\gcd(f(x),x^k-1)=g(x)$, we will use Proposition \ref{proposition:close_zero_gap} exactly $k-d-1$ times. The result of this process will be a new polynomial $f(x)$ equal to $g(x)$ times $k-d-1$ linear polynomials, and $f(x)$ will have the property that $w(f(x))=0$.

To use Proposition \ref{proposition:close_zero_gap}, we must have at least $k$ distinct nonzero $\alpha\in\Fp$. Furthermore, these $\alpha$ \underline{cannot} be roots of $\frac{x^k-1}{g(x)}$. If $\alpha$ were a root of $\frac{x^k-1}{g(x)}$, then $\gcd((x-\alpha)\cdot g(x),x^k-1)$ would have degree greater than $d$. Since $\Fp$ has $p-1$ nonzero elements, we need $p-1\ge k+M$ to satisfy the assumptions of Proposition \ref{proposition:close_zero_gap} and thus we need $p>k+M$.

The result of using Proposition \ref{proposition:close_zero_gap} exactly $k-d-1$ times is a degree $k-1$ polynomial $f(x)=b_{k-1}x^{k-1}+\dots+b_1x+b_0$ with $b_{k-1},\dots,b_0\not\equiv 0\mod p$. Observe that $b_0+\dots+b_{k-1}\equiv 0\mod p$ since $x-1|f(x)$. By Lemma \ref{lemma:associate}, Lemma \ref{lemma:associate_groups}, and Proposition \ref{proposition:main3}, there exists a $k$-gon $[a_0,\dots,a_{k-1}]$ modulo $p$ with monodromy group $G(a_0,\ldots,a_{k-1})\cong C_p^{k-d}\rtimes C_k$.
\end{proof}

Theorem \ref{theorem:rank} proves Theorem \ref{theorem:monodromy_groups} for most $k$ and $p$ as illustrated in the following corollary.

\begin{corollary}\label{corollary:monodromy_groups}
Fix an integer $k\ge 3$. Theorem \ref{theorem:monodromy_groups} is true for primes $p>k+1$.
\end{corollary}

\begin{proof}
Fix $p>k+1$. Suppose that $\gcd(p-1,k)=d$. We claim that $\F_p^\times$ contains exactly $d$ distinct $k$th roots of unity. Observe that $\F_p^\times\cong C_{p-1}\cong \Z\slash (p-1)\Z$. Finding the number of $k$th roots of unity in $\F_p^\times$ is equivalent to finding the number of solutions to $kx\equiv 0\mod p-1$ in $\Z\slash (p-1)\Z$. Since $\gcd(\frac{k}{d},p-1)=1$, we see that the number of solutions to $kx=\frac{k}{d}(dx)\equiv 0\mod p-1$ is the same as the number of solutions to $dx\equiv 0\mod p-1$. Since $d|p-1$, there are $d$ solutions to $dx\equiv 0\mod p-1$ and thus $\F_p^\times$ contains exactly $d$ distinct $k$th roots of unity. The remaining $k$th roots of unity lie in an algebraic extension of $\Fp$. 

In Theorem \ref{theorem:rank}, $M\le d-1$ since the factor $g_{i_j}=x-1$ for some $i_j$. Since $p\ne k+1$ and $\gcd(p-1,k)=d$, we deduce that $p>k+d>k+M$. Thus, Theorem \ref{theorem:monodromy_groups} is true when $p>k+1$.
\end{proof}

\begin{remark}
To prove Theorem \ref{theorem:monodromy_groups}, one need only verify it for integers $k\ge 3$ where $p=k+1$ is prime.
\end{remark}

\subsection{Proving Theorem \ref{theorem:monodromy_groups} for $p=k+1$}
In this section, we prove Theorem \ref{theorem:monodromy_groups} in the remaining cases in which $p=k+1$. 

\begin{remark}
If $p=k+1$ then $x^k-1$ splits completely into linear terms over $\Fp$ since $x^p-x=x(x^k-1)$ is the polynomial whose roots are the elements of $\Fp$.
\end{remark}

\begin{lemma}\label{lemma:deg_d_polynomial}
Suppose $p=k+1$ is an odd prime. Let $d|k$ with $d>1$. There exists a polynomial $x^{d}-a\in \Fp[x]$ with no roots in $\Fp$.
\end{lemma}
\begin{proof}
Since $\Fp^\times$ is a cyclic group under multiplication, let $a$ be a generator of this cyclic group. We claim $x^{d}-a$ has no roots in $\Fp$. Suppose $x^{d}-a$ had a root in $\Fp$. This would imply that there exists an element $b\in\Fp$ satisfying $b^{d}=a$. However, this would imply that $a^{k/d}=(b^d)^{k/d}=b^k=1$, a contradiction with the fact that the order of $a$ under multiplication is $k$.
\end{proof}

\begin{proposition}\label{proposition:small_d}
Suppose $p=k+1$ is an odd prime. Further suppose $0<d<\frac{k}{2}$. There exists a $k$-gon $[a_0,\dots,a_{k-1}]$ modulo $p$ with monodromy group $G(a_0,\ldots,a_{k-1})=C_p^{k-d}\rtimes C_k$.
\end{proposition}

\begin{proof}
By Lemma \ref{lemma:deg_d_polynomial}, there exists a polynomial $x^{k/2}-a$ that has no linear factors in $\Fp$. Thus, $\gcd(x^{k/2}-a,x^k-1)=1$. We need to produce a polynomial $g(x)$ of degree $\frac{k}{2}-1$ so that $w(g(x))=0$ and the $\gcd(g(x),x^k-1)$ has degree $d$. If we can find such a $g(x)$, then $h(x)=(x^{k/2}-a)\cdot g(x)$ has degree $k-1$, the $\gcd(h(x),x^k-1)$ has degree $d$, and $w(h(x))=0$.

Consider $(x-1)^{k/2-d}$ whose coefficients are nonzero modulo $p$. We need to find a sequence of distinct elements $\alpha_i\in\Fp$ so that if we set $g(x)=(x-1)^{k/2-d}\displaystyle\prod_{i=1}^{d-1}{(x-\alpha_i)}$ then $w(g(x))=0$. We proceed by induction. Suppose we have already found $j$ distinct elements $\alpha_i\in \Fp$ so that $\tilde{g}(x)=(x-1)^{k/2-d}\displaystyle\prod_{i=1}^j{(x-\alpha_i)}$ and $w(\tilde{g}(x))=0.$ How many choices for $\alpha_{j+1}$ are there? By Proposition \ref{proposition:close_zero_gap}, since $\deg(\tilde{g})=\frac{k}{2}-d+j$, we need more than $\frac{k}{2}-d+j$ choices to select $\alpha_{j+1}$ so that $w(\tilde{g}\cdot (x-\alpha_{j+1}))=0$. We also remove $j$ possible nonzero elements of $\Fp$ from consideration when we choose $\alpha_{j+1}$ to ensure all $\alpha_i$ are distinct. Since $j<d<\frac{k}{2}$, we see that $\frac{k}{2}-d+j<k-j$. Thus, by the pigeon hole principle, there exists a nonzero $\alpha_{j+1}$ so that the $\alpha_i$ are distinct for $1\le i\le j+1$ and $w(\tilde{g}\cdot (x-\alpha_{j+1}))=0$.

By induction, we have shown there exists a polynomial $g(x)=(x-1)^{k/2-d}\displaystyle\prod_{i=1}^{d-1}{(x-\alpha_i)}$ where the $\alpha_i$ are distinct and $w(g(x))=0$. Now, let $h(x)=g(x)\cdot (x^{k/2}-a)$. We see that $\deg(h(x))=k-1$, the $\gcd(h(x),x^k-1)$ has degree $d$, and $w(h(x))=0$.

By Lemma \ref{lemma:associate}, Lemma \ref{lemma:associate_groups}, and Proposition \ref{proposition:main3}, there exists a $k$-gon $[a_0,\dots,a_{k-1}]$ modulo $p$ with monodromy group $G(a_0,\ldots,a_{k-1})\cong C_p^{k-d}\rtimes C_k$.
\end{proof}

\begin{proposition}\label{proposition:large_d}
Suppose $k\ge 3$ and $p=k+1$ is prime. Further suppose $\frac{k}{2}< d<k$. There exists a $k$-gon $[a_0,\dots,a_{k-1}]$ modulo $p$ with monodromy group $G(a_0,\ldots,a_{k-1})=C_p^{k-d}\rtimes C_k$.
\end{proposition}
\begin{proof}
Since $k$ is even, observe that $x^{k/2}+1$ divides $x^k-1$. Let $S$ be the set of roots of $x^{k/2}+1$ in $\Fp$. Choose a set $T=\{\alpha_1,\dots,\alpha_{d-k/2}\}\subset \F_p^\times$ so that the $\alpha_i$ are distinct, $\alpha_1=1$, and $\alpha_i\not\in S$ for all $i$.

Setting $\tilde{g}(x)=\displaystyle\prod_{i=1}^{d-k/2}{(x-\alpha_i)}$, observe that $\tilde{g}(x)$ divides $x^{k/2}-1$. By Proposition \ref{proposition:zero_spread}, we see that $w(\tilde{g}(x))<\frac{k}{2}-(d-\frac{k}{2})=k-d<\frac{k}{2}$. Now, we want to use Proposition \ref{proposition:close_zero_gap} exactly $k-d-1$ times to find $\beta_j$ in $\Fp$ so that $g(x)=\displaystyle\prod_{i=1}^{d-k/2}{(x-\alpha_i)}\cdot \prod_{j=1}^{k-d-1}{(x-\beta_j)}$ and $w(g(x))=0$ and each $\beta_j\in S\cup T$. If we have at least $\frac{k}{2}$ eligible distinct nonzero elements of $\Fp$, we can use Proposition \ref{proposition:close_zero_gap} exactly $k-d-1$ times. Since there are $d$ nonzero elements in $S\cup T$ and $d>\frac{k}{2}$, we can use Proposition \ref{proposition:close_zero_gap} to select our $\beta_j$. The result of using Proposition \ref{proposition:close_zero_gap} these $k-d-1$ times is the polynomial $g(x)=\displaystyle\prod_{i=1}^{d-k/2}{(x-\alpha_i)}\cdot \prod_{j=1}^{k-d-1}{(x-\beta_j)}$ which has the properties that $w(g(x))=0$ and each $\beta_j\in S\cup T$.

Now, let $h(x)=g(x)\cdot (x^{k/2}+1)$. We see that $\deg(h(x))=k-1$, that $\gcd(h(x),x^k-1)$ has degree $d$, and that $w(h(x))=0$. By Lemma \ref{lemma:associate}, Lemma \ref{lemma:associate_groups}, and Proposition \ref{proposition:main3}, there exists a $k$-gon $[a_0,\dots,a_{k-1}]$ modulo $p$ with monodromy group $G(a_0,\ldots,a_{k-1})\cong C_p^{k-d}\rtimes C_k$.
\end{proof}

Now, we proceed with the proof of Theorem \ref{theorem:monodromy_groups}.

\begin{proof}[Proof of Theorem \ref{theorem:monodromy_groups}]
The case where $p>k+1$ was proven in Corollary \ref{corollary:monodromy_groups}. Now consider the case when $p=k+1$ is an odd prime. If $1\le d\le k-1$, we claim there exists a $k$-gon modulo $p$ with monodromy group $C_p^{k-d}\rtimes C_k$. The case where $d<\frac{k}{2}$ was proven in Proposition \ref{proposition:small_d} and the case where $d>\frac{k}{2}$ was proven in Proposition \ref{proposition:large_d}. The case where $d=\frac{k}{2}$ is a consequence of Proposition \ref{proposition:reducing_rank} because $\frac{k}{2}$ divides $k$. Thus, the proof of Theorem \ref{theorem:monodromy_groups} is complete.
\end{proof}

\section{Results for Composite $n$}\label{section:composite_results}



In this section, we will prove several results about monodromy groups when $n$ is composite relying heavily on the theory of algebraic polygons from Section \ref{section:algebraic_polygons}. This first proposition shows that you can combine $k$-gons with relatively prime moduli to create a new $k$-gon whose monodromy group is closely related to the monodromy groups of the initial $k$-gons.

\begin{proposition}\label{proposition:combine_groups_kgon}
Suppose that $[a_0,\dots,a_{k-1}]$ and $[b_0,\dots,b_{k-1}]$ represent $k$-gons modulo $n_1$ and $n_2$ respectively where $\gcd(n_1,n_2)=1$. Suppose their respective monodromy groups are $N_1\rtimes C_k$ and $N_2\rtimes C_k$. Then there exists a $k$-gon $[c_0,\dots,c_{k-1}]$ modulo $n_1n_2$ with monodromy group $(N_1\times N_2)\rtimes C_k$. 
\end{proposition}

\begin{proof}
This proposition is an immediate consequence of Proposition \ref{proposition:combine_groups}, Lemma \ref{lemma:associate}, and Lemma \ref{lemma:associate_groups}.
\end{proof}

Here is an example of the use of Proposition \ref{proposition:combine_groups_kgon}.

\begin{example}
Consider the quadrilateral $[a_0,a_1,a_2,a_3]=[1,4,4,1]$ which has modulus $n_1=5$. The monodromy group of $D(1,4,4,1)$ is $C_5^2\rtimes C_4$. Also consider the quadrilateral $[b_0,b_1,b_2,b_3]=[2,3,4,3]$ which has modulus $n_2=6$. The monodromy group of $D(2,3,4,3)$ is $C_6^2\rtimes C_4$. We can solve a system of four congruences modulo $5\cdot 6=30$. Observe that if we set $[c_0,c_1,c_2,c_3]=[26, 9, 4,21]$ then we have $c_i\equiv a_i\mod 5$ and $c_i\equiv b_i\mod 6$. We see that $c_0+c_1+c_2+c_3=2\cdot 30$. If this had not been the case, we could have modified the coefficients using Lemma \ref{lemma:associate} and Lemma \ref{lemma:associate_groups} without changing the monodromy group. Finally, one can compute that the monodromy group of $D(26,9,4,21)$ is $C_{30}^2\rtimes C_4\cong (C_5^2\times C_6^2)\rtimes C_4$.
\end{example}

You can use Proposition \ref{proposition:project_group} to project a $k$-gon modulo $n_1n_2$ to an \emph{algebraic} $k$-gon modulo $n_1$. However, this proposition does \emph{not} guarantee that the new algebraic $k$-gon will have a $k$-gon associate as illustrated in the following example.

\begin{example}
Consider the polygon $[c_0,c_1,c_2]=[1,1,4]$ modulo $6$ which has monodromy group $(C_6\times C_2)\rtimes C_3$. Consider the reduction $c_i\equiv a_i\mod 2$ to obtain $[a_0,a_1,a_2]=[1,1,0]$. The monodromy group of $[1,1,0]$ modulo $2$ is $C_2^2\rtimes C_3$. However, there do not exist any $3$-gons modulo $2$.
\end{example}

The above example illustrates how we must understand monodromy groups of algebraic polygons, and not polygons, in order to classify all possible monodromy groups for $k$-gons modulo composite $n$.

\begin{proposition}\label{proposition:composite_groups}
Fix an abelian group $N$ and a positive integer $n=\prod{p_j^{x_j}}$ where the $p_j$ are distinct primes. There exists a $k$-gon $[c_0,\dots,c_{k-1}]$ modulo $n$ with monodromy group $N\rtimes C_k$ if and only if there exist algebraic $k$-gons $[a_0^{(j)},\dots,a_{k-1}^{(j)}]$ modulo $p_j^{x_j}$ with monodromy groups $(N\slash p_{j}^{x_j}N)\rtimes C_k$ and for every $0\le i\le k-1$ there exists some $j$ for which $a_i^{(j)}\not\equiv 0\mod p_j^{x_j}$.
\end{proposition}
\begin{proof}
If $[c_0,\dots,c_{k-1}]$ is a $k$-gon with the desired monodromy group $N\rtimes C_k$, then the forward direction of the proof follows immediately from Proposition \ref{proposition:project_group} and the fact that $c_i\not\equiv 0\mod n$ for all $i$.

Suppose there exist algebraic $k$-gons $[a_0^{(j)},\dots,a_{k-1}^{(j)}]$ modulo $p_j^{x_j}$ with monodromy groups $(N\slash p_{j}^{x_j}N)\rtimes C_k$ and for every $0\le i\le k-1$ there exists some $j$ for which $a_i^{(j)}\not\equiv 0\mod p_j^{x_j}$. The reverse direction of the proof follows from Proposition \ref{proposition:combine_groups}, Lemma \ref{lemma:associate}, and Lemma \ref{lemma:associate_groups}.
\end{proof}

\begin{remark}
The condition that $a_i^{(j)}\not\equiv 0\mod p_j^{x_j}$ in Proposition \ref{proposition:composite_groups} is satisfied if at least one of the algebraic $k$-gons $[a_0^{(j)},\dots,a_{k-1}^{(j)}]$ is an actual $k$-gon. This is sufficient but not necessary.
\end{remark}

Proposition \ref{proposition:composite_groups} translates the problem of understanding the monodromy groups of all algebraic $k$-gons to the problem of understanding monodromy groups for algebraic $k$-gons with prime power moduli.

\begin{example}
There does not exist a $3$-gon modulo $35$ with monodromy group $N\rtimes C_3$ where $N\cong C_{35}$ or where $N\cong C_{35}\times C_7$. Suppose there were such a $3$-gon $[c_0,c_1,c_2]$ modulo $35$. Then the projection of $[c_0,c_1,c_2]$ modulo $5$ (using Proposition \ref{proposition:project_group}) would have monodromy group $7N\rtimes C_3\cong (N\slash 5N)\rtimes C_3$ which is isomorphic to $C_5\rtimes C_3$ in both the case where $N\cong C_{35}$ and $N\cong C_{35}\times C_7$. However, $C_5\rtimes C_3$ is not a possible monodromy group for any algebraic $3$-gon modulo $5$ by Proposition \ref{proposition:possible_groups}.
\end{example}


\subsection{Triangular Billiards Surfaces}
One well-known property of the Smith Normal Form for $\Z$ is summarized in the following lemma.
\begin{lemma}[Proposition 8.1, \cite{MR09}]\label{lemma:snf_minors}
If $d_1,\dots,d_k$ are the elementary divisors of the Smith Normal Form of a matrix $A$ over $\Z$, then $d_1\cdots d_j$ is equal to the gcd of the determinants of all $j\times j$ minors of the matrix $A$. 
\end{lemma}

This property allows us to reprove Corollary \ref{corollary:triangle} using a method that will extend to the higher $k$-gons.

\begin{proof}[Proof of Corollary \ref{corollary:triangle}]
Consider the arbitrary rational triangle with angles $\left(\dfrac{a_0\pi}{n},\dfrac{a_1\pi}{n},\dfrac{a_2\pi}{n}\right)$, where the $a_i$ are positive integers, $a_0+a_1+a_2=n$, and $\gcd(a_0,a_1,a_2,n)=1$. The normal subgroup $N$ of the associated monodromy group is represented by the column span of $C=\begin{bmatrix}
a_0 & a_1 & a_2 \\
 a_1 & a_2 & a_0 \\
a_2 & a_0 & a_1 
\end{bmatrix}$
 over $\mathbb{Z}/n\mathbb{Z}$. Observe that 

$$C=\begin{bmatrix}
a_0 & a_1 & a_2 \\
 a_1 & a_2 & a_0 \\
a_2 & a_0 & a_1 
\end{bmatrix}
=
\begin{bmatrix}
1 & 0 & 0 \\
0 & 1& 0 \\
-1 & -1 & 1
\end{bmatrix}
\begin{bmatrix}
a_0 & a_1 & 0 \\
 a_1 & a_2 & 0 \\
0 & 0 & 0 
\end{bmatrix}
\begin{bmatrix}
1 & 0 & -1 \\
0 & 1 & -1\\
0 & 0 &1
\end{bmatrix}.
$$

The elementary divisors of $C$ are the same as the elementary divisors of $C'=\begin{bmatrix}
a_0 & a_1 & 0 \\
 a_1 & a_2 & 0 \\
0 & 0 & 0 
\end{bmatrix}$. Using Lemma \ref{lemma:snf_minors}, we deduce that $d_1=\gcd(a_0,a_1,a_2,n)=1$. By looking at the $2\times 2$ minors of $C'$, we further deduce that $d_1d_2=d_2=\gcd(a_0a_2-a_1^2,n)$. It then follows from Theorem 2 that the monodromy group of the $(a_0,a_1,a_2)$ triangle is 
$$
  (C_n\times C_{n/\alpha})\rtimes C_3,  
$$
where $d_2=\alpha=\gcd(n,a_0a_2-a_1^2)$. 
\end{proof}

Although Corollary \ref{corollary:triangle} gives a formula for computing the monodromy group of the dessin drawn on a triangular billiards surface, it does not specify which monodromy groups can arise. The following theorem classifies the monodromy groups of all rational triangular billiards surfaces modulo $n$.

\begin{theorem}\label{theorem:all_triangle_groups}
Fix $n\in \N$ with $n> 3$. The set of possible monodromy groups for triangles modulo $n$ includes precisely those groups of the form $(C_n\times C_{n/\alpha})\rtimes C_3$ where $\alpha|n$ and $\alpha=3^i\prod_{j}{p_{j}^{n_j}}$ where the $p_j$ are primes congruent to $1$ modulo $3$, $i\in\{0,1\}$, and $n_j\ge 0$. If $n=3$, the only possible monodromy group is $C_3\rtimes C_3$.
\end{theorem}

The proof of this theorem utilizes results from algebraic number theory. Use any introductory graduate book on the topic, such as \cite{L14}, as a reference.

\begin{proof}
    Recall that the monodromy group associated to the triangle $(a_0,a_1,a_2)$ modulo $n$ is $(C_n\times C_{n\slash \alpha})\rtimes C_3$ where $\alpha=\gcd(a_0a_2-a_1^2,n).$ What values can $a_0a_2-a_1^2$ take modulo $n$?

    Observe that $a_2\equiv -a_0-a_1\pmod n$. Hence, $a_0a_2-a_1^2\equiv a_0(-a_0-a_1)-a_1^2\equiv -(a_0^2+a_0a_1+a_1^2)\mod n.$ Further observe that $a_0^2+a_0a_1+a_1^2=N(a_0-a_1\zeta_3)$ where $\zeta_3$ is a third root of unity and $N$ is the norm map from $\Z[\zeta_3]$ to $\Z$. So we can answer the question about the possible values of $\alpha$ by asking what values are in the image of the norm map. However, there are some restrictions on $a_0$ and $a_1$. Since $\gcd(a_0,a_1,a_2,n)=1$ and $a_0+a_1+a_2=n$, we deduce that $\gcd(a_0,a_1,n)=1$. Hence, if $a_0$ and $a_1$ have a common factor greater than $1$, that factor does not divide $n$. Therefore, to find a triangle modulo $n$ with monodromy group $(C_n\times C_{n\slash \alpha})\rtimes C_3$, we must find an ideal $(a_0-a_1\zeta_3)$ in $\Z[\zeta_3]$ with the properties that $\gcd(N(a_0-a_1\zeta_3),n)=\alpha$ and $\gcd(a_0,a_1,n)=1$. 
    
    The fact that the norm map is multiplicative will allow us to answer the question by examining ideals with norm of prime power order. Since ideals factor uniquely as products of prime ideals in $\Z[\zeta_3]$, suppose the ideal $(a_0-a_1\zeta_3)=\prod{\frak{p}_j^{n_j}}$ where the $\frak{p}_j$ are distinct prime ideals in $\Z[\zeta_3]$. If $\frak{p}_j=(b_0-b_1\zeta_3)$ then $\gcd(b_0,b_1,n)=1$. If $\gcd(b_0,b_1,n)\ne 1$, then $\gcd(a_0,a_1,n)\ne 1$. Secondly, if $p^{n_j}|\gcd(N(a_0-a_1\zeta_3),n)$ one of the following three situations must arise:
    \begin{enumerate}
        \item $\frak{p}^{n_j/2}=(p)^{n_j/2}$ is in the factorization of the ideal $(a_0-a_1\zeta)$ if $\frak{p}$ is an inert prime with $N(\frak{p})=p^2$.
        \item $\frak{p}_1^{x}\frak{p}_2^{n_j-x}$ is in the factorization of the ideal $(a_0-a_1\zeta_3)$ if $\frak{p}_1$ and $\frak{p}_2$ are the two primes above $(p)$ in $\Z[\zeta_3]$. In this case, $N(\frak{p}_1)=N(\frak{p}_2)=p$.
        \item $\frak{p}^{n_j}$ is in the factorization of the ideal $(a_0-a_1\zeta)$ if $\frak{p}$ is a ramified prime with $N(\frak{p})=p$.
    \end{enumerate}

   To summarize, we want to know if, when $p$ is a prime dividing $n$, does there exist an ideal $(b_0-b_1\zeta_3)$ satisfying $N(b_0-b_1\zeta_3)=p^{n_j}$ with $p^{n_j}|n$ and $\gcd(b_0,b_1,n)=1$?

    First consider a prime $p\equiv 2\mod 3$. Observe that the ideal $(p)\subset \Z[\zeta_3]$ is an inert prime ideal that has norm $p^2$. Hence, $p$ is not in the range of the norm map. If $b_0-b_1\zeta_3\in\Z[\zeta_3]$ has norm $p^{n_j}$ then the ideal generated by $b_0-b_1\zeta_3$ has the property $(b_0-b_1\zeta_3)=(p^{n_j/2})$ since ideals factor uniquely as products of prime ideals in $\Z[\zeta_3]$. Hence, $p$ divides $b_0$ and $b_1$, which implies $p\nmid n$. Hence, $p\not\equiv 2\mod 3$.

    Now consider a prime $p\equiv 1\mod 3$. There is a prime ideal $(y-z\zeta_3)$ of norm $p$ since the ideal $(p)$ splits in $\Z[\zeta_3]$. Note that $\gcd(y,z)=1$ since $N(y-z\zeta_3)=y^2+yz+z^2=p$. Set $(y-z\zeta_3)^{n_j}=(b_0-b_1\zeta_3)$. Observe that the ideal $(b_0-b_1\zeta_3)$ is an ideal with norm $p^{n_j}$. Now, we deduce $\gcd(b_0,b_1)=1$ from the fact that ideals factor uniquely in $\Z[\zeta_3]$. Since $N(b_0-b_1\zeta_3)=p^{n_j}$, the only factor they could have in common is $p$. But if $p|b_0$ and $p|b_1$ then the ideal $(p)$ would divide $(y-z\zeta_3)^{n_j}$, which is a contradiction since the ideal $(p)$ factors as a product of two distinct prime ideals of norm $p$, namely $(y-z\zeta_3)$ and $(y-z\zeta_3^2)$. Clearly, $(y-z\zeta_3^2)$ is not in the unique factorization of $(y-z\zeta_3)^{n_j}$. Hence, $\gcd(b_0,b_1)=1$. Therefore, if $p\equiv 1\mod 3$ is a prime dividing $n$, then there exist $b_0,b_1$ with $\gcd(b_0,b_1)=1$ and $N(b_0-b_1\zeta_3)=p^{n_j}$.

    Now consider the case when $p=3$. The unique prime ideal of norm $3$ in $\Z[\zeta_3]$ is $(1-\zeta_3)$. If $N(b_0-b_1\zeta_3)=3^{i}$ where $i>1$ then the ideal $(3)$ would divide $(b_0-b_1\zeta_3)$ since the ideal $(1-\zeta_3)^2=(3)$. Since ideals have unique prime ideal factorizations in $\Z[\zeta_3]$, we would have $3|b_0$ and $3|b_1$, a contradiction. Hence, when $p=3$, the only ideal $(b_0-b_1\zeta_3)$ satisfying $N(b_0-b_1\zeta_3)=3^i$ with $3^i|n$ and $\gcd(b_0,b_1,n)=1$ occurs when $i\in\{0,1\}$.

    Using the multiplicative property of the norm map, if $\alpha=3^i\prod_{j}{p_{j}^{n_j}}$ divides $n$ where the $p_j$ are primes congruent to $1$ modulo $3$, $i\in\{0,1\}$, and $n_j\ge 0$, then there exist positive integers $a_0,a_1$ with $\gcd(a_0,a_1,n)=1$, and $\gcd(a_0a_2-a_1^2,n)=\alpha$ if $a_2=n-a_0-a_1.$ To use Lemma \ref{lemma:associate}, we must verify that $a_0,a_1,a_2\not\equiv  0\mod n$. 

    Assume $\alpha\ne 1$. By way of contradiction, assume one of the $a_i\equiv 0\mod n$. Without loss of generality, assume $a_2\equiv 0$.
     In this case, $a_0\equiv -a_1\mod n$. Thus, $a_0^2+a_0a_1+a_1^2\equiv a_0^2\mod n$. Hence, $\gcd(N(a_0-a_1\zeta_3),n)=\gcd(a_0^2+a_0a_1+a_1^2,n)=\gcd(a_0^2,n).$ Since, $\gcd(a_0,a_1,n)=\gcd(a_0,-a_0,n)=1$, then $\gcd(N(a_0-a_1\zeta_3),n)=\gcd(a_0^2,n)=1$, a contradiction.
    
    Thus, if $\alpha\ne 1$, we can use Lemma \ref{lemma:associate} to adjust $[a_0,a_1,a_2]$ so that it is a geometric $3$-gon modulo $n$ without altering the $\gcd$'s above. Thus by Corollary \ref{corollary:triangle}, we have obtained the required monodromy group when $\alpha\ne 1$.

    Now consider the case when $\alpha=1$. Instead of showing $a_i\not\equiv 0\mod n$ in the above construction, we instead find explicit geometric triangles with monodromy group $(C_n\times C_n)\rtimes C_3$. If $3\nmid n$, then consider the triangle $[1,1,n-2]$. Observe that $\gcd(a_0^2+a_0a_1+a_2^2,n)=\gcd(3,n)=1$. Thus, $[1,1,n-2]$ has monodromy group $(C_n\times C_n)\rtimes C_3$ when $3\nmid n$. Now consider the case when $3|n$. Consider the triangle $[\frac{n}{3}-1,\frac{n}{3},\frac{n}{3}+1]$. This is a geometric triangle when $n>3$. Observe that $a_0^2+a_0a_1+a_1^2=(\frac{n}{3}-1)^2+(\frac{n}{3}-1)\frac{n}{3}+(\frac{n}{3})^2=1-n+\frac{n^2}{3}$. Since $3|n$, we see that $a_0^2+a_0a_1+a_1^2\equiv 1\mod n$. Thus $\gcd(a_0^2+a_0a_1+a_1^2,n)=1$ and the monodromy group of $[\frac{n}{3}-1,\frac{n}{3},\frac{n}{3}+1]$ is $(C_n\times C_n)\rtimes C_3$. In the case when $n=3$, there is only one geometric triangle, $[1,1,1]$, which has monodromy group $C_3\rtimes C_3$.
\end{proof}

The following example illustrates how Theorem \ref{theorem:all_triangle_groups} can be used to classify the possible monodromy groups modulo a composite number $n$.

\begin{example}
If $n=81$, there are only two possible monodromy groups. The triangle $[1,2,78]$ has associated monodromy group $(C_{81}\times C_{81})\rtimes C_3$ and the triangle $[1,1,79]$ has associated monodromy group $(C_{81}\times C_{27})\rtimes C_3$. However, there does not exist a triangle with associated monodromy group $(C_{81}\times C_9)\rtimes C_3$ or $(C_{81}\times C_3)\rtimes C_3$ or $(C_{81})\rtimes C_3$.
\end{example}

\subsection{Quadrilateral Billiards Surfaces}
One can also use Lemma \ref{lemma:snf_minors} to produce an analogue of Corollary \ref{corollary:triangle} in the quadrilateral case.

\begin{proposition}\label{proposition:quadrilateral}
Suppose that $[a_0,a_1,a_2,a_3]$ represents a $4$-gon modulo $n$. Let $G(a_0, a_1, a_2, a_3)$ be the monodromy group of the dessin $D(a_0,a_1,a_2,a_3)$ drawn on the quadrilateral billiards surface $X(a_0,a_1,a_2,a_3)$. Then
$$
G(a_0, a_1, a_2, a_3) \cong (C_n \times C_{\frac{n}{d_2}}\times C_{\frac{n}{d_3}}) \rtimes C_4.
$$
where $$d_2=\gcd(a_0a_2-a_3^2,a_0a_1-a_2a_3, a_0^2-a_2^2,a_1a_3-a_2^2,a_0a_3-a_1a_2,a_0a_2-a_1^2,n)$$ and 
$$
d_3=\begin{cases}
\gcd(\frac{(a_0+a_2)((a_0+a_1)^2+(a_1+a_2)^2)}{d_2},n)& if\ d_2\ne n\\
n& if\ d_2=n.
    \end{cases}
$$
\end{proposition}
\begin{proof}
The normal subgroup $N$ of the associated monodromy group is represented by the column span of $C=\begin{bmatrix}
a_0 & a_1 & a_2 &a_3\\
 a_1 & a_2 & a_3 &a_0\\
a_2 & a_3& a_0&a_1\\
a_3& a_0& a_1&a_2
\end{bmatrix}$
 over $\mathbb{Z}/n\mathbb{Z}$. Let $\tilde{a}_3=-a_0-a_1-a_2$. Consider the matrix $C'=\begin{bmatrix}
a_0 & a_1 & a_2 &\tilde{a}_3\\
 a_1 & a_2 & \tilde{a}_3 &a_0\\
a_2 & \tilde{a}_3&a_0&a_1\\
\tilde{a}_3& a_0& a_1&a_2
\end{bmatrix}.$ Observe that $C\equiv C'\mod n$ and thus they have the same elementary divisors modulo $n$. We will proceed by finding the elementary divisors of $C'$ over $\Z$ and then reducing them modulo $n$ to get the elementary divisors of $C'$. Let $d_1,d_2,d_3,d_4$ be the elementary divisors of $C$ and let $\tilde{d}_1,\tilde{d}_2,\tilde{d}_3,\tilde{d}_4$ be the elementary divisors of $C'$. Since $\gcd(a_0,a_1,a_2,\tilde{a}_3,n)=\gcd(a_0,a_1,a_2,a_3,n)=1$, the $\gcd$ of the one by one minors is $1$. Hence, $d_1=\tilde{d}_1=1$ by Lemma \ref{lemma:snf_minors}. 

 Observe that 
$$
C'=\begin{bmatrix}
a_0 & a_1 & a_2 &\tilde{a}_3\\
 a_1 & a_2 & \tilde{a}_3 &a_0\\
a_2 & \tilde{a}_3& a_0&a_1\\
\tilde{a}_3& a_0& a_1&a_2
\end{bmatrix}\\
=
\begin{bmatrix}
1 & 0 & 0 &0\\
0 & 1& 0 &0\\
0&0&1&0\\
-1&-1&-1&1
\end{bmatrix}
\begin{bmatrix}
a_0 & a_1 & a_2 &0\\
 a_1 & a_2 & \tilde{a}_3 &0\\
a_2 & \tilde{a}_3& a_0&0\\
0& 0& 0&0
\end{bmatrix}
\begin{bmatrix}
1 & 0 &0& -1 \\
0 & 1 &0& -1\\
0 & 0 &1&-1\\
0&0&0&1
\end{bmatrix}.
$$
Thus the elementary divisors of $C'$ are the same modulo $n$ as the elementary divisors of $$C''=\begin{bmatrix}
a_0 & a_1 & a_2 &0\\
 a_1 & a_2 & \tilde{a}_3 &0\\
a_2 & \tilde{a}_3& a_0&0\\
0& 0& 0&0
\end{bmatrix}.$$ Hence, $d_4=\tilde{d}_4=0$. To compute $d_2$, we compute the gcd of the 2 by 2 minors of $C''$ of which there are only $9$ that are nonzero. Three of the minors are duplicates, thus leaving us with $6$. These minors are
$\{a_0a_2-\tilde{a}_3^2,a_0a_1-a_2\tilde{a}_3, a_0^2-a_2^2,a_1\tilde{a}_3-a_2^2,a_0\tilde{a}_3-a_1a_2,a_0a_2-a_1^2 \}$. Using Lemma \ref{lemma:snf_minors}, we obtain $d_2=\gcd(\tilde{d}_2,n)=\gcd(a_0a_2-a_3^2,a_0a_1-a_2a_3, a_0^2-a_2^2,a_1a_3-a_2^2,a_0a_3-a_1a_2,a_0a_2-a_1^2,n)$.

Lastly, $\tilde{d}_3$ will be equal to the third elementary divisor of $C'$ which is the same as the third elementary divisor of $ \begin{bmatrix}
a_0 & a_1 & a_2 \\
 a_1 & a_2 & \tilde{a}_3 \\
a_2 & \tilde{a}_3& a_0\\
\end{bmatrix}.$ By Lemma \ref{lemma:snf_minors}, we know that $\tilde{d}_2\tilde{d}_3=\det \begin{bmatrix}
a_0 & a_1 & a_2 \\
 a_1 & a_2 & \tilde{a}_3 \\
a_2 & \tilde{a}_3& a_0\\
\end{bmatrix}=a_0^2a_2+2a_1a_2\tilde{a}_3-a_2^3-a_0\tilde{a}_3^2-a_0a_1^2=-(a_0+a_2)((a_0+a_1)^2+(a_1+a_2)^2)$. Hence, $\tilde{d}_3=\frac{(a_0+a_2)((a_0+a_1)^2+(a_1+a_2)^2)}{\tilde{d}_2}$ provided $\tilde{d}_2\ne 0$. If $\tilde{d}_2=0$ then $\tilde{d}_3=0$. Therefore, $d_3=\gcd(\tilde{d}_3,n)=\gcd(\frac{(a_0+a_2)((a_0+a_1)^2+(a_1+a_2)^2)}{d_2},n)$ unless $d_2=n$ in which case $d_3=n$.
\end{proof}

\section{Future Directions}


There are many questions that naturally arose in the study of monodromy groups of dessin drawn on rational billiards surfaces. Here are some possible future questions to investigate.

\begin{question}
Throughout this paper, we used Proposition \ref{proposition:adjust_angles}, Lemma \ref{lemma:associate}, and Lemma \ref{lemma:associate_groups} many times to produce a polygon with the same monodromy group as a particular algebraic polygon. Using Lemma \ref{lemma:associate}, we can produce an associate {\it convex} polygon in the case where the modulus $n=p$ is prime and $p\ge k$. It is natural to ask if $G$ is the monodromy group of a $k$-gon modulo $n$, is it the monodromy group of a convex $k$-gon modulo $n$?
\end{question}

\begin{question} How can one generalize Theorem \ref{theorem:monodromy_groups} to primes $p\le k$? For $p\le k$, a monodromy group attained by an algebraic $k$-gon may not be attainable by a $k$-gon. For example, $x^6-1=(x-1)^2(x^2+x+1)^2$ modulo $2$. Thus, there exist algebraic $6$-gons modulo $2$ with monodromy groups $C_2\rtimes C_6$, $C_2^2\rtimes C_6$, $C_2^3\rtimes C_6$, $C_2^4\rtimes C_6$, and $C_2^5\rtimes C_6$. However, there is only one $6$-gon modulo $2$, namely $[3,1,1,1,1,1]$, which has monodromy group $C_2\rtimes C_6$. 
\end{question}

\begin{question}
Can one generalize Proposition \ref{proposition:quadrilateral} to $k$-gons where $k>4$?
\end{question}

\begin{question}
In Theorem \ref{theorem:all_triangle_groups}, we classified which groups appear as the monodromy group of a triangle. Can one prove an analogous result for the monodromy groups that arise for an arbitrary $k$-gon?
\end{question}


\bibliographystyle{plain}
\bibliography{main}\newpage

\end{document}